\documentclass{amsart}
\usepackage{amssymb,amscd,verbatim,latexsym}
\usepackage{enumerate, caption}

\usepackage[l2tabu, orthodox]{nag}
\usepackage[mathscr]{eucal}
\usepackage{tikz}
\usetikzlibrary{decorations.pathreplacing}

\definecolor{darkgreen}{rgb}{0,0.6,0}
\definecolor{darkred}{rgb}{0.7,0,0}
\definecolor{darkblue}{rgb}{0,.2,.7}
\definecolor{berndpink}{rgb}{.6,.2,.6}
 \usepackage[%
   colorlinks, 
  citecolor=darkgreen, linkcolor=darkblue,
   pdfpagelabels=true,
   unicode=true
 ]{hyperref}

\usepackage[normalem]{ulem}

\usepackage{color}

\newcommand\define{\mathrel{:=}}
\newcommand\ede{\define}
\newcommand\seq{=}

\newcommand\esssup{\mathop{\mathrm{ess}\kern0.08em\textrm{-}\kern0.08em\mathrm{sup}}}

\newcommand\II{\mathrm{I}\hskip-.3mm\mathrm{I}}

\renewcommand\ne[1]{|#1|}
\newcommand\cunif{c_{\mathrm{unif}}}

\newcommand\datver[1]{\def\datverp
{\par\boxed{\boxed{\text{Version: #1; Run: \today}}}}}\datver{0.1}

\newcommand{\dist}{\operatorname{dist}}

\newcommand{\CC}{\mathbb C}
\newcommand{\HH}{\mathbb{H}}

\newcommand{\NN}{\mathbb N}

\newcommand{\RR}{\mathbb R}

\newcommand{\CIc}{{\mathcal C}^{\infty}_{\text{c}}}
\newcommand\pa{{\partial}}

\newcommand{\loc}{\operatorname{loc}}

\newcommand{\vol}{\operatorname{vol}}
\newcommand{\dvol}{\operatorname{dvol}}

\newcommand{\maC}{\mathcal C}
\newcommand{\maD}{\mathcal D}

\newcommand\<{\langle}
\renewcommand\>{\rangle}

\newcommand{\rinj}{\mathop{r_{\mathrm{inj}}}}
\newcommand\partialDM{\partial_{\kern-.04em D}\kern-.08em M}
\newcommand\partialNM{\partial_{\kern-.04em N}\kern-.08em M}
\newcommand\partialD{\partial_{\kern-.04em D}\kern-.08em}
\newcommand\partialN{\partial_{\kern-.04em N}\kern-.08em}

\newtheorem{theorem}{Theorem}[section]
\newtheorem{proposition}[theorem]{Proposition}
\newtheorem{corollary}[theorem]{Corollary}

\newtheorem{lemma}[theorem]{Lemma}
\theoremstyle{definition}
\newtheorem{definition}[theorem]{Definition}

\theoremstyle{remark}
\newtheorem{remark}[theorem]{Remark}
\newtheorem{notation}[theorem]{Notations}
\newtheorem{example}[theorem]{Example}

\author[B. Ammann]{Bernd Ammann} \address{B. Ammann, Fakult\"at f\"ur
  Mathematik, Universit\"at Regensburg, 93040 Regensburg, Germany}
\email{bernd.ammann@mathematik.uni-regensburg.de}

\author[N. Gro\ss e]{Nadine Gro\ss e} \address{N. Gro\ss e,
  Mathematisches Institut, Universit\"at Freiburg, 79104 Freiburg,
  Germany} \email{nadine.grosse@math.uni-freiburg.de}

\author[V. Nistor]{Victor Nistor} \address{Universit\'{e} de Lorraine,
  UFR MIM, Ile du Saulcy, CS 50128, 57045 METZ, France and
  Math. Dept., Penn State Univ, PA 16802, USA}
\email{victor.nistor@univ-lorraine.fr}

\thanks{B.A. has been partially supported by SFB 1085 Higher Invariants,
Regensburg, funded by the DFG. 
  V.N. has been partially supported by
  ANR-14-CE25-0012-01 (SINGSTAR).\\
Manuscripts available from \textbf{http:{\scriptsize
    //}iecl.univ-lorraine.fr{\scriptsize
    /}$\tilde{}$Victor.Nistor{\scriptsize /}}}

\date\today

\begin{document}

\title[Well-posedness of the Laplacian and bounded
  geometry]{Well-posedness of the Laplacian on manifolds with boundary
  and bounded geometry}

\begin{abstract}
Let $M$ be a Riemannian manifold with a smooth boundary.  The main
question we address in this article is: ``When is the Laplace-Beltrami
operator $\Delta\colon H^{k+1}(M)\cap H^1_0(M) \to H^{k-1}(M)$, $k\in
\NN_0$, invertible?'' We consider also the case of mixed boundary
conditions. The study of this main question leads us to the class of
manifolds with boundary and bounded geometry introduced by Schick
(Math. Nach. 2001). We thus begin with some needed results on the
geometry of manifolds with boundary and bounded geometry. Let
$\partialDM \subset \pa M$ be an open and closed subset of the
boundary of $M$. We say that $(M, \partialDM)$ has \emph{finite width}
if, by definition, $M$ is a manifold with boundary and bounded
geometry such that the distance $\dist(x, \partialDM)$ from a point $x
\in M$ to $\partialDM \subset \pa M$ is bounded uniformly in $x$ (and
hence, in particular, $\partialDM$ intersects all connected components
of $M$).  For manifolds $(M, \partialDM)$ with finite width, we prove
a Poincar\'e inequality for functions vanishing on $\partialDM$, thus
generalizing an important result of Sakurai (Osaka J. Math, 2017). The
Poincar\'e inequality then leads, as in the classical case to results
on the spectrum of $\Delta$ with domain given by mixed boundary
conditions, in particular, $\Delta$ is invertible for manifolds $(M,
\partialDM)$ with finite width. The bounded geometry assumption then
allows us to prove the well-posed\-ness of the Poisson problem with
mixed boundary conditions in the higher Sobolev spaces $H^s(M)$, $s
\ge 0$.
\end{abstract}

\maketitle
{\hypersetup{linkcolor=black}\tableofcontents}

\section{Introduction}

This is the first in a sequence of papers devoted to the spectral and
regularity theory of differential operators on a suitable non-compact
manifold with boundary~$M$ using analytic and geometric methods.

Let $\Delta \define d^* d \geq 0$ be the (geometer's) Laplace operator
acting on functions, also called the Laplace-Beltrami operator. Let
$k$ be a non-negative integer.  The \textbf{Main Question} of this
paper is:
\begin{quotation}
 ``Is the operator $\Delta \colon H^{k+1}(M) \cap H^1_0(M) \to
  H^{k-1}(M)$ invertible?''
\end{quotation}
The range of $k$ needs to be specified each time, and deciding the
range for $k$ is part of the problem, but in this paper we are mainly
concerned with the cases $k=0$ and $k \in \mathbb{N}_0 := \{0, 1, 2,
\ldots \}$. Also part of the problem is, of course, to choose the
``right definition'' for the relevant function spaces. The Main
Question of this paper turns out, in fact, to be mostly a geometric
question, involving the underlying properties of our manifold~$M$.  It
involves, in particular a geometric approach to the Poincar\'e
inequality and to the regularity of the solution $u$ of $\Delta u =
f$. Most of our results on the Laplacian extend almost immediately to
uniformly strongly elliptic operators, however, in order to keep the
presentation simple, we concentrate in this paper on the
Laplace-Beltrami operator. The general case will be discussed in
  \cite{AGN3}. Note also that our question is different (but related)
  to that of the presence of zero in the spectrum of the Laplacian
  \cite{LottZero}.

For $k = 0$, it is easy to see that the invertibility of $\Delta$ in
our Main Question is equivalent to the Poincar\'e inequality for
$L^2$-norms (see, for instance, Proposition \ref{prop.isomorphism}).
This answers completely our question for $k = 0$. Our main result
(proved in the more general framework of mixed boundary conditions) is
that our Main Question has an affirmative answer for all $k$ {\em
  whenever $M$ has finite width.} As a consequence, we obtain results
on the spectral theory of the Laplacian with suitable mixed boundary
conditions on manifolds with bounded geometry as well as on the
regularity of the solutions of equations of the form $\Delta u = f$.

Let us formulate our results in more detail. Let $(M,g)$ be a smooth
$m$-dimensional Riemannian manifold with \emph{smooth boundary and
  bounded geometry} (see Definition~\ref{def_bdd_geo} and above).  Our
manifolds are assumed to be paracompact, but not necessarily second
countable. See Remark \ref{def.mfd}.  \emph{It is important in
  applications \underline{not} to assume $M$ to be connected.}  We
denote the boundary of~$M$ by~$\pa M$, as usual, and we assume that we
are given a {\em disjoint union decomposition}
\begin{align}\label{eq.decomposition}
 \pa M \ = \ \partialDM \amalg \partialNM\,,
\end{align}
where $\partialDM$ or $\partialNM$ are (possibly empty) open subsets
of $\pa M$ and $\amalg$ denotes the disjoint union, as usual. We shall
say that the boundary of $M$ is {\em partitioned}. Of course, both
$\partialDM$ and $\partialNM$ will also be closed.

Also, we shall say that the pair $(M, \partialDM)$ has \emph{finite
  width} if the distance from any point of $M$ to $\partialDM$ is
\emph{bounded uniformly} on~$M$ and $\partialDM$ intersects all
connected components of $M$ (Definition~\ref{def.finite.width}). In
particular, if $(M, \partialDM)$ has finite width, then $\partialDM$
is not empty.  The simplest example is $M=\mathbb R^{n-1}\times
[0,T]\subset \mathbb R^n$ with $\partial_D M= \mathbb R^{n-1}\times
\{0\}$ and $\partial_N M= \mathbb R^{n-1}\times \{1\}$.

Before stating our main result, Theorem \ref{thm_Poin_intro}, it is
convenient to first state the Poincar\'e inequality that is used in
its proof.

\begin{theorem}[Poincar\'e inequality]\label{thm_Poin_intro}
Let $M$ be an $m$-dimensional smooth Riemannian manifold with smooth,
partitioned boundary $\partial M = \partialDM \amalg \partialNM$.
Assume that $(M,\partialDM)$ has finite width. Then, for every $p\in
[1,\infty]$, $(M, \partialDM)$ satisfies the $L^p$-Poincar\'e
inequality, that is, there exists $0 < c=c_{M,p} < \infty$ such that
\begin{equation*} 
   \Vert f\Vert _{L^p(M)} \leq c \big( \Vert f\Vert _{L^p(\partialDM)} +
   \Vert d f\Vert _{L^p(M)} \big)
\end{equation*}
for all $f \in W^{1,p}_{\loc}(M)$.
\end{theorem}

In a nice, very recent paper \cite{Sakurai}, Sakurai has proved this
Poincar\'e inequality for $\partialNM = \emptyset$, and $p = 1$. His
proof (posted on the Arxive preprint server shortly before we posted
the first version this paper) can be extended to our case (once one
takes care of a few delicate points; see Remarks~\ref{rem_sakurai} and
\ref{rem_extend}).  In particular, one can sharpen our Poincar\'e
inequality result by relaxing the bounded-geometry condition by
replacing it with a lower Ricci bound condition and a bound on the
second fundamental form; see Remark~\ref{rem_sakurai}. However, we do
need the `full' bounded geometry assumption for the higher regularity
part of our results, so we found it convenient to consider the current
slightly simplified setting.

For a manifold with boundary
and bounded geometry, we set
\begin{equation}\label{eq.def.H1D}
 H^1_D(M) \define \{\, f \in H^1(M)\, \vert \ f = 0 \mbox{ on }
 \partialDM\, \}
\end{equation}
(note that this definition makes sense, due to the trace theorem for
such manifolds, see Theorem~\ref{thm.trace}) and
\begin{equation}\label{eq.def.cmdm}
 c_{M, \partialDM} \ede \inf \bigl\{\, t \in \RR \,
 \vert\ \|f\|_{L^2(M)} \le t \|d f\|_{L^2(M)},\, (\forall) \, f \in
 H^1_D(M)\, \bigr\},
\end{equation}
with the agreement that $c_{M, \partialDM} = \infty$ if the set on the
right hand side is empty.  Thus $c_{M, \partialDM} > 0$ is the best
constant in the Poincar\'e inequality for $p = 2$ and vanishing
boundary values (see Definition~\ref{def.pPoincare}). Clearly $c_{M,
  \partialDM} > 0$. The definition of $c_{M, \partialDM}$ implies that
$c_{M, \partialDM} := \infty$ if and only if $(M, \partialDM)$ does
not satisfy the Poincar\'e inequality for $p = 2$.

Let $\nu$ be an outward unit vector at the boundary of our manifold
with bounded geometry $M$ and let $\pa_\nu$ be the directional
derivative with respect to $\nu$. Let $d$ be the de Rham
  differential, $d^*$ be its formal adjoint and $\Delta := d^*d$ be
  the scalar Laplace operator. Our main result is the following
well-posed\-ness result.

\begin{theorem} \label{thm.main1}
Let $(M, \partialDM)$ be a manifold with boundary and bounded geometry
and $k\in \mathbb{N}_0 := \{0, 1, 2, \ldots \}$.  Then $\Delta$ with
domain
\begin{equation*}
 \maD(\Delta) \define H^2(M) \cap \{ u = 0 \mbox{ on } \partialDM
 \mbox{ and } \pa_\nu u = 0 \mbox{ on } \partialNM \}
\end{equation*}
is self-adjoint on $L^2(M)$ and, for all $\lambda\in \RR$ with
$\lambda < \gamma := (1 + c_{M, \partialDM}^2)^{-1}$, we have
isomorphisms
\begin{equation*}
 \Delta - \lambda \colon H^{k+1}(M) \cap \maD(\Delta) \to H^{k-1}(M),
 \quad k \geq 1.
\end{equation*}
If $(M, \partialDM)$ has finite width, then $\gamma > 0$.  
\end{theorem}

The purely Dirichlet version of this Theorem is contained in
Theorem~\ref{thm.Hm}; for the general version, see Theorem
\ref{thm.Hm2}. See Subsection~\ref{ssec.4.3} for further details. One
obtains a version of the isomorphism in the Theorem also for
$k=0$. The example of a smooth domain $\Omega \subset \RR^n$ that
coincides with the interior of a smooth cone outside a compact set
shows that some additional assumptions on $(M, \partialDM)$ are
necessary in order for $\Delta$ to be invertible.

Our well-posed\-ness result is of interest in itself, as a general
result on analysis on non-compact manifolds, but also because it has
applications to partial differential equations on singular and
non-compact spaces, which are more and more often studied. Examples
are provided by analysis on curved space-times in general relativity
and conformal field theory, see \cite{BaerGinoux, BaerWafo, gerardBG,
  GerardWrochna, KRS15, SchrohePhysics, TataruSurv15} and the
references therein. An important example of manifolds with bounded
geometry is that of asymptotically hyperbolic manifolds, which play an
increasingly important role in geometry and physics
\cite{ammannGrosse16b, bouclet2006,delay2008, moroianu2010}.  It was
shown in \cite{mueller.nardmann:15} that every Riemannian manifold is
conformal to a manifold with bounded geometry.  Moreover, manifolds of
bounded geometry can be used to study boundary value problems on
singular domains, see, for instance, \cite{BMNZ,
  NazarovSokolowskiAA09, CDN12, daugeBook, JerisonSavin, KMR,
  MunnierCusp, NazarovCusp94, NP} and the references therein.

The proof of Theorem \ref{thm.main1} is based on a Poincar\'{e}
inequality on~$M$ for functions vanishing on~$\partialDM$, under the
assumption that the pair $(M, \partialDM)$ has finite width, and on
local regularity results (see Theorems \ref{thm_Poin_intro} and
\ref{thm.regularity}). The higher regularity results is obtained using
a description of Sobolev spaces on manifolds with bounded geometry
using partitions of unity \cite{Grosse.Schneider.2013, TriebelBG} and
is valid without the finite width assumption (we only need that $M$ is
with boundary and bounded geometry for our regularity result to be
true). We also need the classical regularity of the Dirichlet and
Neumann problems for strongly elliptic operators on smooth
domains. The non-compactness of the boundary is dealt with using a
continuity argument.

The paper is organized as follows. Section~\ref{sec.one} is devoted to
preliminary results on manifolds with boundary and bounded
geometry. We begin with several equivalent definitions of manifolds
with boundary and bounded geometry and their basic properties, some of
which are new results; see in particular
Theorem~\ref{theo_crit_bdd_geo}. The subsections~\ref{ssec.inj}
and~\ref{ssec.control} are devoted to some further properties of
manifolds with boundary and bounded geometry that are useful for
analysis. In particular, we introduce and study submanifolds with
bounded geometry and we devise a method to construct manifolds with
boundary and bounded geometry using a gluing procedure, see Corollary
\ref{cor.gluing}.  In Subsection~\ref{ssec.Sobolev} we recall the
definitions of basic coordinate charts on manifolds with boundary and
bounded geometry and use them to define the Sobolev spaces, as for
example in \cite{Grosse.Schneider.2013}.  The Poincar\'e inequality
for functions vanishing on $\partialDM$ and its proof, together with
some geometric preliminaries can be found in Section~\ref{sec.two}.
The last section is devoted to applications to the analysis of the
Laplace operator. We begin with the well-posed\-ness in energy spaces
(if $(M, \partialDM)$ has finite width) and then we prove additional
regularity results for solutions as well the fact that spectrum of
$\Delta$ is contained in the closed half-line $\{ \lambda \in \RR
\vert \lambda \ge \gamma\}$ (see Theorem \ref{thm.main1}). In
particular, the proofs of Theorem~\ref{thm.main1}, as well as that of
its generalization to mixed boundary conditions,
Theorem~\ref{thm.Hm2}, can both be found in this section. Some
  of our results (such as regularity), extend almost without change to
  $L^p$-spaces, but doing that would require a rather big overhead of
  analytical technical background. On the other hand, certain results,
  such as the explicit determination of the spectrum using positivity,
  {\em do not extend} to the $L^p$ setting. See 
  \cite{ammannGrosse16b, kordyukovLp2,sturm_93} and the references therein.

\section{Manifolds with boundary and bounded geometry\label{sec.one}}

In this section we introduce manifolds with boundary and bounded
geometry and we discuss some of their basic properties. Our approach
is to view both the manifold and its boundary as submanifolds of a
larger manifold with bounded geometry, but without boundary. The
boundary is treated as a submanifold of codimension one of the larger
manifold, see Section~\ref{ssec.def.bdd}.  This leads to an
alternative definition of manifolds with boundary and bounded
geometry, which we prove to be equivalent to the standard one, see
\cite{Schick.2001}.

We also recall briefly in Section~\ref{ssec.Sobolev} the definition of the
Sobolev spaces used in this paper.

\subsection{Definition of manifolds with boundary and bounded geometry}
\label{ssec.def.bdd}
\emph{Unless explicitly stated otherwise, we agree that throughout
  this paper, $M$ will be a smooth Riemannian manifold of dimension
  $m$ with smooth boundary $\pa M$, metric $g$ and volume form
  $\dvol_g$.  We are interested in non-compact manifolds~$M$. The
  boundary will be a disjoint union $\pa M = \partialDM \amalg
  \partialNM$, with $\partialDM$ and $\partialNM$ both open (and
  closed) subsets of $\pa M$, as in Equation
  \eqref{eq.decomposition}.}  (That is, the boundary of $M$ is
partitioned.)  Typically, $M$ will be either assumed or proved to be
with (boundary and) bounded geometry.

\begin{remark}[Definition of a manifold]\label{def.mfd}
In textbooks, one finds two alternative definitions of a smooth
manifold. Some textbooks define a manifold as a locally Euclidean,
second countable Hausdorff space together with a smooth structure, see
e.g. \cite[Chapiter~1]{lee:smooth-manifolds}.  Other textbooks, however,
replace second countability by the weaker requirement that the
manifold be a paracompact topological space. The second choice implies
that every connected component is second countable, thus both
definitions coincide if the number of connected components is
countable. Thus, a manifold in the second sense is a manifold in the
first sense if, and only if, the set of connected components is
countable. As an example, consider $\mathbb{R}$ with the discrete
topology. It is not a manifold in the first sense, but it is a
$0$-dimensional manifold in the second sense.  For most statements in
differential geometry, it does not matter which definition is chosen.
However, in order to allow uncountable index sets $I$ in
Theorem~\ref{thm_Poin}, we will only assume that our manifolds are
paracompact (thus we will {\em not require} our manifolds to be second
countable). This is needed in some applications.
 \end{remark}
 
For any $x\in M$, we define 
\begin{align*}
  \maD_x\define \left\{v \in T_xM\mid \text{there exists a geodesic
    $\gamma_v\colon [0, 1] \to M$ with $\gamma_v'(0)=v$}\right\}.
\end{align*}
This set is open and star-shaped in $T_xM$ and
\begin{equation}\label{eq.def.D}
  \maD\define \bigcup_{x\in M}\maD_x
\end{equation}
is an open subset of $TM$. The map $\exp^M\colon \maD\to M$,
$\maD_x\ni v\mapsto \exp^M(v)=\exp^M_x(v) \define  \gamma_v(1)$ is called
the \emph{exponential map}.

If $x, y \in M$, then $\dist(x, y)$ denotes the distance between $x$
and $y$, computed as the infimum of the set of lengths of the paths in
$M$ connecting $x$ to $y$.  If $x$ and $y$ are not in the same
connected component, we set $\dist(x, y)=\infty$. The map
$\dist\colon M\times M\to [0,\infty]$ satisfies the axioms of an extended
metric on $M$. This means that it satisfies the usual axioms of a
metric, except that it takes values in $[0,\infty]$ instead of
$[0,\infty)$.  If $A \subset M$ is a subset, then
\begin{equation}\label{eq.def.Ur}
 U_r(A) \define \{x \in M \mid\, \exists y \in A,\, \dist(x, y) < r\,\}
\end{equation}
will denote the $r$-neighborhood of $A$, that is, the set of points of
$M$ at distance $< r$ to $A$. Thus, if $E$ is a Euclidean space, then
\begin{equation}\label{eq.def.BEr}
 B_r^E(0)\define U_r(\{0\}) \subset E
\end{equation}
is simply the ball of radius
$r$ centered at $0$. 

We shall let $\nabla=\nabla^Y$ denote the Levi-Civita connection on a
Riemannian manifold $Y$ and $R^Y$ denote its curvature.  Now let $Y$
be a submanifold of $M$ with its induced metric.  Recall the maximal
domain $\maD \subset TM$ of the exponential map from Equation
\eqref{eq.def.D}. Let $\maD^\perp \define \maD \cap \nu_Y$ be the
intersection of $\maD$ with the normal bundle $\nu_Y$ of $Y$ in $M$
and
\begin{align}\label{eq.def.exp.perp}
 \exp^\perp \define \exp|_{\maD^\perp}\colon \maD^\perp \to M,
\end{align}
be the \emph{normal exponential map}. By choosing a local unit normal
vector field $\nu$ to $Y$ to locally identify $\maD^{\perp}$ with a
subset of $Y\times \RR$, we obtain $exp^\perp(x,t) \define
\exp^M_x(t\nu_x)$. As~$\maD$ is open in~$TM$, we know that
$\maD^\perp$ is a neighborhood of $Y \times \{0\}$ in $Y \times \RR$.

\begin{definition}\label{def_ttly_bdd_curv}
We say that a Riemannian manifold $(M,g)$ has \emph{totally bounded
  curvature}, if its curvature $R^M$ satisfies
\begin{align*}
  \Vert \nabla^k R^M \Vert_{L^\infty} < \infty\qquad\text{for all
  } k \ge 0.
\end{align*}
where $\nabla$ is the Levi-Civita connection on $(M,g)$, as before.
\end{definition}

Recall the definition of the ball $B_r^E$ of Equation \eqref{eq.def.BEr}. 
The injectivity radius of $M$ at $p$, respectively, the {\em injectivity
radius} of $M$ are then
\begin{align*}
\begin{gathered}
  \rinj(p) \define \sup\{ r \mid \exp_p \colon B_r^{T_pM}(0) \to M
  \text{ is a diffeomorphism onto its image} \}\\
 \rinj(M) \define \inf_{p \in M} \rinj(p) \in [0,\infty].
\end{gathered}
\end{align*}
Recall also following classical definition in the case $\partial
M=\emptyset$.

\begin{definition}
A Riemannian manifold without boundary $(M,g)$ is said to be of
\emph{bounded geometry} if $M$ has \emph{totally bounded curvature}
and $\rinj(M)>0$.
\end{definition}

This definition cannot be carried over in a straightforward way to
manifolds with boundary, as manifolds with non-empty boundary always
have $\rinj(M)=0$. We use instead the following alternative approach.

Let $N^n$ be a submanifold of a Riemannian manifold $M^m$. Let
$\nu_N\subset TM|_N$ be the normal bundle of $N$ in $M$. Informally,
the second fundamental form of $N$ in $M$ is defined as the map
\begin{equation}\label{eq.def.II}
 \II\colon TN\times TN\to \nu_N, \quad \II(X,Y)\define \nabla^M_X Y -
 \nabla^N_X Y.
\end{equation}

Let us specialize this definition to the case when $N$ is a
hypersurface in $M$ (a submanifold with $\dim N=\dim M - 1$) assumed
to carry a globally defined normal vector field $\nu$ of unit length,
called a \emph{unit normal field}. Then one can identify the normal
bundle of $N$ in $M$ with $N\times \RR$ using $\nu$, and hence, the
second fundamental form of $N$ is simply a smooth family of symmetric
bilinear maps $\II_p\colon T_pN\times T_pN\to \RR$, $p\in N$. In
particular, we see that $\II$ defines a smooth tensor. See
\cite[Chapter~6]{doCarmo} for details.

\begin{definition}\label{hyp_bdd_geo} 
Let $(M^{m},g)$ be a Riemannian manifold of bounded geometry with a
hypersurface $N^{m-1}\subset M$ with a unit normal field $\nu$ to $N$.
We say that $N$ is a {\em bounded geometry} hypersurface if the
following conditions are fulfilled:
\begin{enumerate}[(i)]
\item $N$ is a closed subset of $M$;
\item\label{bg.sub} $(N, g|_N)$ is a manifold of bounded geometry;
\item The second fundamental form $\II$ of $N$ in $M$ and all its
  covariant derivatives along $N$ are bounded. In other words: 
  \begin{equation*}
   \Vert \left(\nabla^N\right)^k \II\Vert _{L^\infty}\leq C_k \ 
   \mbox{ for all }\    k\in \NN_0;
  \end{equation*}
\item There is a number $\delta>0$ such that $\exp^\perp\colon N\times
  (-\delta,\delta)\to M$ is injective.
\end{enumerate}
\end{definition}

See Definition \ref{def.gen.bg} for the definition of arbitrary
codimension submanifolds with bounded geometry.  We prove in
Section~\ref{ssec.control} that Axiom \eqref{bg.sub} of
Definition~\ref{hyp_bdd_geo} is redundant. In other words it already
follows from the other axioms. We keep Axiom~\eqref{bg.sub} in our
list in order to make the comparison with the definitions in
\cite{Schick.2001} and \cite{Grosse.Schneider.2013} more apparent.
  
\begin{definition}\label{def_bdd_geo}  
A Riemannian manifold~$(M,g)$ with (smooth) boundary has \emph{bounded
  geometry} if there is a Riemannian manifold $(\widehat M, \hat{g})$ with
bounded geometry satisfying
\begin{enumerate}[(i)]
\item $\dim \widehat M=\dim M$
\item $M$ is contained in $\widehat M$, 
in the sense that there is an isometric embedding $(M,g)\hookrightarrow 
(\widehat{M}, \hat{g})$
\item $\partial M$ is a bounded geometry hypersurface in $\widehat M$.
\end{enumerate}
\end{definition}

As unit normal vector field for $\partial M$ we choose the outer unit
normal field. Similar definitions were considered by \cite{AmannAnis,
  AmannFunctSp, Browder60, Schick.2001}. We will show in
Section~\ref{ssec.control} that our definition coincides with
\cite[Definition~2.2]{Schick.2001}. In this section we also discuss
further conditions, which are equivalent to the conditions in
Definition~\ref{def_bdd_geo}.  

Note that if $M$ is a manifold with boundary and bounded geometry,
then each connected component of $M$ is a complete metric space.  To
simplify the notation, we say that a Riemannian manifold is
\emph{complete}, if any of its connected components is a complete
metric space. The classical Hopf-Rinow theorem
\cite[Chapiter~7]{doCarmo} then tells us that a Riemannian manifold is
complete if, and only if, all geodesics can be extended indefinitely.

\begin{example}\label{ex.Lie.Man}
An important example of a manifold with boundary and bounded geometry
is provided by Lie manifolds with boundary \cite{antonini, sobolev,
  aln1}. 
\end{example}

Recall the definition of the sets $U_R(N)$ of Equation \eqref{eq.def.Ur}.
For the Poincar\'e inequality, we shall also need to assume that $M
\subset U_R(\partialDM)$, for some $R > 0$, and hence, in particular,
that $\partialDM\neq \emptyset$. We formalize this in the concept of
``manifold with finite width.''

\begin{definition}\label{def.finite.width}
  Let $(M,g)$ be a Riemannian manifold with boundary $\partial M$ and
  $A \subset \partial M$. We say that $(M, A)$ has
  \emph{finite width} if:
\begin{enumerate}[(i)]
\item $(M,g)$ is a manifold with boundary and bounded geometry, and
\item\label{cond.fw} $M \subset U_R(A)$, for some $R < \infty$.
\end{enumerate}
If $A = \pa M$, we shall also say that~$M$ \emph{has finite
  width}.
\end{definition}

Note that axiom \eqref{cond.fw} implies that $A$ intersects
all connected components of $M$, as $\dist(x,y) := \infty$ if $x$ and
$y$ are in different components of $M$. Also, recall that the meaning
of the assumption that $M \subset U_R(A)$ in this definition
is that every point of $M$ is at distance \emph{at most} $R$ to
$A$.

\begin{remark}[Hausdorff distance]
Recall that for subsets $A$ and $B$ of a metric space $(X,d)$ one
defines the \emph{Hausdorff distance} between $A$ and $B$ as
\begin{equation*}
  d_H(A,B) \ede \inf\{R>0\,|\, A\subset U_R(B)\text{ and }B\subset
  U_R(A)\}\in[0,\infty].
\end{equation*}
(See Equation \eqref{eq.def.BEr} for notation.) The assignment
$(A,B)\mapsto d_H(A,B)$ defines an extended metric on the set of all
closed subsets of $X$. Again, the word ``extended'' indicates that
$d_H$ satisfies the usual axioms of a metric, but that it takes values
in $[0,\infty]$ instead of $[0,\infty)$. See
  \cite[9.11]{berger:geometry} or \cite[\S 7/14]{rinow} for a
  definition and discussion of the Hausdorff distance on the set of
  compact subsets.  In this language, a pair $(M, A)$ as above
  has finite width if, and only if, $d_H(M, A)<\infty$.  The
  definition of the Hausdorff distance implies
  $d_H(\emptyset,M)=\infty$.
\end{remark}

\begin{example}
A very simple example of a manifold with finite width is obtained as
follows. We consider $M \define \Omega \times K$ where $\Omega$ is a
smooth, compact Riemannian manifold with smooth boundary and $K$ is a
Riemannian manifold with bounded geometry ($K$ has no boundary). Then
$M$ is a manifold with boundary and bounded geometry. Let
$\partialD \Omega$ be a union of connected components of $\pa \Omega$
that intersects every connected component of $\Omega$. Let $\partialDM
\define \partialD \Omega \times K$. Then $(M, \partialDM)$ is a
manifold with finite width.

On the other hand, let $\Omega \subset \RR^N$ be an open subset with
smooth boundary. We assume that the boundary of $\Omega$ coincides
with the boundary of a cone outside some compact set. Then $(\Omega,
\pa \Omega)$ is a manifold with boundary and bounded geometry, but it
does not have finite width. See also Example \ref{ex.Omega}.
\end{example}

An alternative characterization of manifolds with boundary and bounded
geometry is contained in the following theorem, that will be proven in
Section~\ref{ssec.control} using some preliminaries on the injectivity
radius that are recalled in Section~\ref{ssec.inj}

\begin{theorem}\label{theo_crit_bdd_geo}  
    Let $(M,g)$ be a Riemannian manifold with smooth boundary and
    curvature $R^M = (\nabla^M)^2$ and let $\II$ be the second
    fundamental form of the boundary $\pa M$ in $M$. Assume
    that:
\begin{enumerate}
  \item[(N)\,] There is $r_\partial>0$ such that
  \begin{align*} 
     \partial M \times [0, r_\partial) \to M,\ (x, t) \mapsto 
     \exp^\perp(x,t) \define \exp_x(t\nu_x)
  \end{align*}
  is a diffeomorphism onto its image.
  \item[(I)\ ] There is $\rinj(M) > 0$ such that for all $r \leq
    \rinj(M)$ and all $x\in M\setminus U_{r}(\partial M)$, the
    exponential map $\exp_x \colon B_r^{T_xM}(0)\subset T_xM \to M$
    defines a diffeomorphism onto its image.
  \item[(B)\,]  For every $k \geq0$, we have 
 \begin{align*}
  \Vert \nabla^k R^M\Vert_{L^\infty} < \infty \text{\ \ \ and\ \ \ }
  \Vert (\nabla^{\partial M})^k \II\Vert_{L^\infty}< \infty \,.
 \end{align*}
\end{enumerate}
Then $(M,g)$ is a Riemannian manifold with boundary and bounded
geometry in the sense of Definition~\ref{def_bdd_geo}.
\end{theorem}

\begin{remark}\label{rem.low.ass}
The theorem implies in particular, that our definition of a manifold
with boundary and bounded geometry coincides with the one given by
Schick in \cite[Definition~2.2]{Schick.2001}. According to Schick's
definition, a manifold with boundary has bounded geometry if it
satisfies (N), (I) and (B) and if the boundary  itself has positive
injectivity radius. One of the statements in the theorem is that (N),
(I) and (B) imply that the boundary has bounded geometry.
\end{remark}

\begin{remark}
In \cite{botvinnik.mueller:p15} Botvinnik and M\"uller defined
manifolds with boundary with $(c,k)$-bounded geometry. Their
definition differs from our definition of manifolds with boundary and
bounded geometry in several aspects, in particular they only control
$k$ derivatives of the curvature.
\end{remark}

\subsection{Preliminaries on the injectivity radius}\label{ssec.inj}

We continue with some technical results on the injectivity radius.
Let $(M,g)$ be a Riemannian manifold without boundary and $p\in M$. We
write $\rinj(p)$ for the injectivity radius of $(M,g)$ at $p$ (that is
the supremum of all $r$ such that $\exp\colon B_r^{T_pM}(0)\subset
T_pM\to M$ is injective, as before).

We define the \emph{curvature radius} of $(M,g)$ at $p$  by
\begin{align*}
  \rho \define \sup \bigl\{ r>0\mid \exp_p \text{ is defined on $B_{\pi
      r}^{T_pM}(0)$ and $|\text{sec}|\leq 1/r^2$ on } B_{\pi
    r}^{M}(p)\bigr\}\,,
\end{align*}
where $\text{sec}$ is the sectional curvature and ``$|\text{sec}|\leq
1/r^2$ on $B_{\pi r}^{M}(p)$'' is the short notation for
``$|\text{sec}(E)|\leq 1/r^2$ for all planes $E\subset T_qM$ and all
$q\in B_{\pi r}^M(p)$''. It is clear that the set defining 
$\rho = \rho_p$ is not empty, and hence $\rho = \rho_p \in (0, \infty]$
is well defined.
Standard estimates in Riemannian geometry, see e.g. \cite{doCarmo},
show that the exponential map is an immersion on $B_{\pi
  \rho}^{M}(p)$. In other words, no point in $B_{\pi \rho}^{M}(p)$ is
conjugate to $p$ with respect to geodesics of length $<\pi\rho$.  

Let us notice that if there is a geodesic $\gamma$ from $p$ to $p$ of
length $2\delta>0$, then it is obvious that $\rinj(p)\leq \delta$.
The converse is true if $\delta$ is small:
\begin{lemma}[{\cite[Proposition~III.4.13]{sakai:96}}]
  \label{lem.loop.exists}
  If $\rinj(p)<\pi\rho$, then there is geodesic of length $2\rinj(p)$
  from $p$ to $p$.
\end{lemma}

We will use the following theorem, due to Cheeger, see
\cite{cheeger:thesis}. We also refer to \cite[Sec. 10,
  Lemma~4.5]{petersen:98} for an alternative proof, which easily
generalizes to the version below.

\begin{theorem}[Cheeger]\label{thm_cheeger}
Let $(M,g)$ be a Riemannian manifold with curvature satisfying 
$|R^M|\leq K$. Let $A\subset
M$. We assume that there is an $r_0>0$ such that $\exp_p$ is defined
on $B_{r_0}^{T_pM}(0)$ for all $p\in A$.  Let $\rho>0$. Then
$\inf_{p\in A} \rinj(p)>0$ if, and only if, $\inf_{p\in A}
\vol(B_\rho^M(p))>0$.
\end{theorem}

\subsection{Controlled submanifolds
  are of bounded geometry}\label{ssec.control}

The goal of this subsection is to prove
Theorem~\ref{theo_crit_bdd_geo}. For further use, some of the Lemmata
below are formulated not just for hypersurfaces, such as $\pa M$, but 
also for submanifolds of higher codimension.

Let again $N^n$ be a submanifold of a Riemannian manifold $M^m$ with
second fundamental form $\II$ and normal exponential map $\exp^\perp$
as recalled in \eqref{eq.def.exp.perp} and \eqref{eq.def.II}. For
$X\in T_pN$ and $s\in \Gamma(\nu_N)$, one can decompose $\nabla_X s$
in the $T_pN$ component (which is given by the second fundamental
form) and the normal component $\nabla^{\perp}_X s$. This definition
gives us a connection $\nabla^\perp$ on the bundle $\nu_N\to N$.  We
write $R^\perp$ for the associated curvature.

\begin{definition}\label{def.controlled}
Let $M$ be a Riemannian manifold without boundary. We say that a
closed submanifold $N\subset M$ with second fundamental form $\II$ is
\emph{controlled} if $\Vert (\nabla^{N})^k \II\Vert_{L^\infty} <
\infty$ for all $k\geq 0$ and there is $r_\partial>0$ such that
\begin{align*}
   \exp^\perp\colon V_{r_\pa}(\nu_N) \to M
\end{align*}
is injective, where $V_{r}(\nu_N)$ is the set of all vectors in
$\nu_N$ of length $<r$.
\end{definition}

We will show that the geometry of a controlled submanifold is
``bounded.'' Thus, ``controlled'' should be seen here as an auxiliary
label and we will change the name ``controlled submanifold'' to
\emph{``bounded geometry submanifold.''}

\begin{lemma}\label{lem_extend_M}
For every Riemannian manifold $(M,g)$ with boundary and bounded
geometry, there is a complete Riemannian manifold $\widehat M$ {\em
  without boundary} such that:
\begin{enumerate}[(i)]
 \item $\Vert \nabla^k R^{\hat{M}}\Vert_{L^\infty} <
\infty$, for all $k\geq 0$;
 \item $M\to \widehat M$ is an isometric
embedding, and
 \item $\pa M$ is a controlled submanifold of
$\widehat{M}$.
\end{enumerate}
\end{lemma}

\begin{proof}
The metric $(\exp^\perp)^*g$ on $\partial M \times [0, r_\partial)$ is
  of the form $(\exp^\perp)^*g=h_r + dr^2$ where $r\in [0,
 r_\partial)$ and where $h_r$, $r\in [0, r_\partial)$, is a family
of metrics on $\partial M$ such that $(\partial/\partial r)^k h_r$ is
a bounded tensor for any $k\in \NN$.  Using a cut-off argument, it is
possible to define $h_r$ also for $r\in (-1-r_\partial,0)$ in such a
way that all $(\partial/\partial r)^k h_r$ are bounded tensors and
$h_r=h_{-1-r}$ for all $r\in (-1-r_\partial,r_\partial)$.  An
immediate calculation shows that then $(\partial M \times
(-1-r_\partial, r_\partial),h_r + dr^2)$ has bounded curvature~$R$,
and all derivatives of~$R$ are bounded. We then obtain $\widehat M$ by
gluing together two copies of $M$ together along $\partial M \times
(-1-r_\partial, r_\partial)$. Obviously the curvature and all its
derivatives are bounded on $\widehat M$.
\end{proof}

\begin{lemma}\label{lem_widehatM_inj}
Let $(M,g)$ be a complete Riemannian manifold with boundary $\partial
M$ satisfying conditions (N), (I), and (B) introduced in
Theorem~\ref{theo_crit_bdd_geo}. Then the injectivity radius of the
manifold $\widehat{M}$ as constructed in the proof of the last lemma
is positive.
\end{lemma}

\begin{proof}
We use (I) for $r\define \frac{1}{2}\min \{ \rinj(M), r_\pa\}$ to
obtain that $\inf_{q\in M\setminus U_r(\pa M)} \rinj(q)>0$. It thus
remains to show that $\inf_{q\in \pa M\times [-1-r,r]} \rinj(q)>0$.
Let $q=(x,t)\in \pa M \times [-1-r,r]$. We define the diffeomorphism
$f_q\colon \pa M\times (t-r, t+r)\to \pa M\times (0,2r)\subset M$,
$(y,s)\mapsto (y,s-t+r)$. Then the operator norms $\Vert (df_q)\Vert$
and $\Vert (df_q)^{-1}\Vert$ are bounded on their domains by a bound
$C_1\geq 1$ that only depends on $(M,g)$ and the chosen extension
$\widehat{M}$, but not on~$q$.  Theorem~\ref{thm_cheeger} for
$\rho=r/C_1$ gives $v>0$ such that $\vol(B_{r/C_1}^M(z))>v$ for all
$z\in M\setminus U_r(\pa M)$.  Together with $B_{r/C_1}^{ M}(f_q(q))
\subset f_q(B_{r}^{\widehat M}(q)) \subset \pa M \times (0,2r)$ we get
$$
 \vol (B_{r}^{\widehat M}(q))\geq C_1^{-m} \vol (f_q(B_{r}^M(q))) \geq
 C_1^{-m} \vol (B_{r/C_1}^M(q))\geq C_1^{-m}v .$$
Using again
Theorem~\ref{thm_cheeger} for $\rho=r$ we obtain the required
statement.
\end{proof}

\begin{lemma}\label{lem_subN}
Let $(M,g^M)$ be a Riemannian manifold without boundary and with
totally bounded curvature (Definition \ref{def_ttly_bdd_curv}).
Let $N\subset M$ be a submanifold with
$\Vert (\nabla^N)^k \II\Vert_{L^\infty} < \infty$ for all $k\geq 0$.
Then
  \begin{enumerate}[(i)]
   \item $\Vert (\nabla^N)^k R^N\Vert_{L^\infty} < \infty$ for all
     $k\geq 0$.
   \item The curvature of the normal bundle of $N$ in $M$ and all its
     covariant derivatives are bounded.  In other words $\Vert
     (\nabla^N)^kR^{\perp}\Vert _{L^\infty}\leq c_k$ for all $k\in
     \NN_0$.
   \end{enumerate}
\end{lemma}

\begin{proof}
 (i) This is \cite[Lemma~4.5]{Grosse.Schneider.2013}.

(ii) For a normal vector field $\eta$ let $W_\eta$ be the Weingarten
  map for $\eta$, i.e.,\ for $X\in T_pN$, let $W_\eta(X)$ be the
  tangential part of $-\nabla_X \eta$.  Thus
  $g^M(\II(X,Y),\eta)=g^M(W_\eta(X),Y)$ for the vector-valued second
  fundamental form $\II$.  The Ricci equation \cite[p. 5]{BCC} states
  that the curvature $R^{\perp}$ of the normal bundle is
\begin{align}\label{codazzi-or-mainardi}
  g^M(R^\perp(X,Y)\eta,\zeta)&= g^M(R^M(X,Y)\eta,\zeta) -
  g^M(W_\eta(X),W_\zeta(Y))\\ &\phantom{= }
  +g^M(W_\eta(Y),W_\zeta(X)), \nonumber
\end{align}
where $X,Y\in T_pN$, and where $\eta,\zeta\in T_pM$ are normal to $N$.
The boundedness of $R^M(X,Y)$ and $\II$ thus implies the boundedness
of of $R^\perp$. In order to bound $(\nabla^N)^kR^\perp$, one has to
differentiate \eqref{codazzi-or-mainardi} covariantly. The difference
between $\nabla$ and $\nabla^N$ provides then additional terms that
are linear in $\II$. Thus, one iteratively sees that
$(\nabla^N)^kR^\perp$ is a polynomial in the variables $\II$,
$\nabla^N\II$, \ldots, $(\nabla^N)^k\II$, $R$, $\nabla R^M$, \ldots,
$\nabla^kR^M$, and thus that it is bounded.
\end{proof}

The total space of the normal bundle $\nu_N$ carries a natural
Riemannian metric. Indeed, the connection $\nabla^\perp$ defines a
splitting of the tangent bundle of this total space into the vertical
tangent space (which is the kernel of the differential of the base
point map that maps every vector to its base point) and the horizontal
tangent space (given by the connection).  The horizontal space of
$\nu_N$ inherits the Riemannian metric of $N$ and the vertical tangent
spaces of $\nu_N$ are canonically isomorphic to the fibers of the
normal bundle $\nu_N$ and they carry the canonical metric. Let
$\pi\colon \nu_N\to N$ be the base point map. Then $\pi$ is a
Riemannian submersion.

Let $F\define \exp^\perp|_{V_{r_1}(\nu_N)}$ (see Definition
\ref{def.controlled} for notation). It is clear from the definition
of the normal exponential map that the bounds on $\nabla^kR^M$ and on
$(\nabla^{N})^k\II$ imply that there is an $r_1\in (0,r_\pa)$ such
that the derivative $dF$ and its inverse $(dF)^{-1}$ are uniformly
bounded.

\begin{lemma}\label{lem_pos_inj_N}
Let $M^m$ be a Riemannian manifold without boundary and of bounded
geometry. Let $N^n$ be a controlled submanifold of $M$ (Definition
\ref{def.controlled}).  Then the injectivity radius of $N$ is
positive.
\end{lemma}

\begin{proof}
The proof is similar to the one of Lemma~\ref{lem_widehatM_inj}.
Let $F\define \exp^\perp|_{V_{r_1}(\nu_N)}$ be as above. 
Assume $\Vert dF\Vert\leq C_2$ and $\Vert (dF)^{-1}\Vert\leq C_2$ with
$C_2\geq 1$ on $V_{r_1}(\nu_N)$. For $q\in N$, we define
$B(q)\define V_{r_1}(\nu_N) \cap \pi^{-1}(B_{r_1}^N(q))$, using the
notations of Definition~ \ref{def.controlled}. Then,
$B^M_{r_1/C_2}(q)\subset F(B(q))$ by the boundedness of
$(dF)^{-1}$. By Theorem~\ref{thm_cheeger} applied to $M$ and to
$\rho=r_1/C_2$, there is a $v>0$, independent of $q$, such that
$\vol(B^M_{r_1/C_2}(q))\geq v$, and thus $\vol (F(B(q)))\geq v$. Thus,
$\vol(B(q))\geq C_2^{-n}v$ where the volume on $B(q)$ is taken with
respect to the natural metric on $\nu_N$.  Then by
Lemma~\ref{lem_subN} (compare to the proof of
\cite[Lemma~4.4]{Grosse.Schneider.2013}), $\vol (B(q)) \leq C_3
r_1^{m-n} \vol (B_{r_1}^N(q))$.  Thus, $\vol (B_{r_1}^N(q))$ is
bounded from below independently on $q\in N$.
Theorem~\ref{thm_cheeger} applied for $\rho=r_1$ then yields that $N$
has a positive injectivity radius.
\end{proof}

This sequence of lemmata gives us the desired result.

\begin{corollary}
  Let $N$ be a controlled submanifold in a Riemannian manifold of
  bounded geometry. Then, $N$ is a manifold of bounded geometry.
\end{corollary}

The next corollary, together with Lemma~\ref{lem_extend_M}, implies
then Theorem~\ref{theo_crit_bdd_geo}.

\begin{corollary}
  Let $M$ be a Riemannian manifold with boundary satisfying (N), (I)
  and (B). Then, $\pa M$ is a manifold of bounded geometry.
\end{corollary}

As we have seen, controlled submanifolds have bounded geometry
  in the sense of Theorem~\ref{theo_crit_bdd_geo}. This suggests the
  following definition.

\begin{definition}\label{def.gen.bg}
A controlled submanifold $N \subset M$ of a manifold with bounded
geometry will be called a \emph{bounded geometry submanifold} of
$M$. 
\end{definition}

See also Eldering's papers \cite{ElderingCR, ElderingThesis}.
In particular a hypersurface with unit normal field is controlled if,
only if, it is a bounded geometry hypersurface as in
Definition~\ref{hyp_bdd_geo}. The following example is sometimes
useful.

\begin{example} \label{ex.Omega} Let $(M_0, g_0)$ be a 
manifold with bounded geometry (thus, without boundary). Let $f\colon
M_0 \to \RR$ be a smooth function such that $df$ is totally bounded
(that is, all covariant derivatives $\nabla^k df$ are uniformly
bounded for $k \geq 0$). The function $f$ itself does not have to be
bounded. Let us check that 
\begin{equation*}
  \Omega_{\pm}(f) \define \{(x, t) \in M\define M_0 \times \RR\vert\,
  \pm (t - f(x) ) \geq 0 \}
\end{equation*}
is {\em a manifold with boundary and bounded geometry} using
Theorem~\ref{theo_crit_bdd_geo}. Most of the conditions of 
Theorem~\ref{theo_crit_bdd_geo} are already fulfilled
since $(M_0, g_0)$ is of bounded geometry; only condition (N) and the
part of condition (B) including the second fundamental form remain to
be checked. We start with the second fundamental form. The second
fundamental form of $N$ is defined by $\II (X,Y)=-g(\nabla^{M}_X \nu,
Y)\nu$ for all vector fields $X,Y$ tangent to $N$ where $g=g_0+dt^2$
where a unit normal vector field of $N\define \pa \Omega_{\pm}(f)$ is
given by
\begin{equation*}
  \nu_{(x, f(x))} \seq (1+ |\text{grad}^g
  f(x)|^2)^{-\frac{1}{2}}(\text{grad}^g f(x) -\partial_t).
\end{equation*}
A vector field tangent to $N$ has the form $X=(X_1,X_2\partial_t)$
with $X_1\in \Gamma(TM_0)$ and $X_2\in C^\infty(\RR)$ with
$df(X_1)=X_2$.  Since the metric $g$ has product structure on
$M_0\times \RR$, the Levi-Civita connection splits componentwise. More
precisely, given the vector fields $X=X_1+X_2\partial_t$ and
$Y=Y_1+Y_2\partial_t$ tangent to $N$ we get
\begin{align*} 
 \II(X,Y)& =- \<Y_1+ Y_2 \partial_t, (\nabla^{M_0}_{X_1}
 +X_2\partial_t)\nu\>\, \nu.
\end{align*}
Since $\nabla^k df$ is bounded for all $k\geq 0$ and $(M,g)$ has
bounded geometry, we get that~$\II$ and all its covariant derivatives
are uniformly bounded. In particular the mean curvature of $N$ is
bounded.

In order to verify condition (N) of Theorem~\ref{theo_crit_bdd_geo} we
first note that due to the product structure on $M= M_0\times \RR$, it
suffices to prove that there are constant $r_\partial, \delta>0$ such
that $\exp^\perp|_{V_x}$ is a diffeomorphism onto its image for all
$x=(y,t)\in N$ and $V_x\define \{ (y',t')\in N\ |\ d_{g_0}(y,y')
<\delta\}$. Since $f$ is bounded, there is $\hat{\delta}>0$ such that
$V_x\subset B_{\hat{\delta}}(x)\subset M$ and together with the
bounded geometry of $M_0$ the existence of positive constants
$r_\partial$ and $\delta$ follow. Hence, altogether we get that
$\Omega_{\pm}(f)$ is a manifold with boundary and bounded geometry.
    
Let us now assume additionally that $g \colon M_0 \to \RR$ has the
same properties (i.e.~$dg$ is totally bounded).  Assume $g < f$. Then
\begin{equation*}
  \Omega(f, g) \define \Omega_{-}(f) \cap \Omega_+(g) = \{(x, t) \in
  M_0 \times \RR\ \vert\, g(x) \leq t \leq f(x) \}
\end{equation*}
is a manifold with boundary. Denote $M \define \Omega(f, g)$ and let
$\partialDM$ be any non-empty union of connected components of $\pa
M$.  Then $(M, \partialDM)$ has bounded geometry if, and only if,
there exists $\epsilon > 0$ such that $f - g \geq \epsilon$. It has
finite width if, and only if, there exists also $R > 0$ such that $R
\geq f - g \geq \epsilon$.  In \cite{Browder60}, manifolds with
boundary and bounded geometry were considered in the particular case
when they were subsets of $\RR^n$. Our criteria for $\Omega(f, g)$ to
be of bounded geometry thus helps reconcile the definitions in
\cite{Browder60} and \cite{Schick.2001}.
\end{example}

The following corollary of Theorem \ref{theo_crit_bdd_geo}, when used
in conjunction with Example \ref{ex.Omega}, yields many examples of
manifolds with boundary and bounded geometry.

\begin{corollary}
\label{cor.gluing}
 Let $M = \bigcup_{i=1}^{N} W_i$ be a Riemannian manifold with boundary 
 with $W_i$ open subsets in $M$. We assume that 
 \begin{enumerate}
  \item[(N')\,] For each $1 \le i \le N$, there exists $r_i>0$ such that
  \begin{align*} 
     (\partial M \cap W_i) \times [0, r_i) \to W_i,\ (x, t) \mapsto 
     \exp^\perp(x,t) \define \exp_x(t\nu_x)
  \end{align*}
  is well defined and a diffeomorphism onto its image.
  \item[(I')\ ] There is $\rinj(M) > 0$ such that for all $r \leq
    \rinj(M)$ and all $x\in M\setminus U_{r}(\partial M)$, there 
    exists $1 \leq i \leq N$ such that $x \in W_i$ and the 
    exponential map $\exp_x \colon B_r^{T_x W_i}(0)\to W_i$
    is well-defined and a diffeomorphism onto its image.
  \item[(B)\,]  For every $k \geq0$ and $i = 1, \ldots, N$, we have 
 \begin{align*}
  \Vert \nabla^k R^{W_i} \Vert_{L^\infty} < \infty \text{\ \ \ and\ \ \ }
  \Vert (\nabla^{\partial W_i})^k \II\Vert_{L^\infty}< \infty \,.
 \end{align*}
\end{enumerate}
Then $(M,g)$ is a Riemannian manifold with boundary and bounded
geometry. 
\end{corollary}

In applications, the subsets $W_i$ will be open subsets of some other
manifolds with boundary and bounded geometry (say of some manifold of
the form $\Omega(f, g)$).  In general, the subsets $W_i$ will not be
with bounded geometry.  If, moreover, the boundary of $M$ is
partitioned: $\pa M = \partialDM \amalg \partialNM$ and $d_H(W_i, W_i
\cap \partialDM ) < \infty$ for all $i$, then $(M, \partialDM)$ will
have finite width.

\subsection{Sobolev spaces via
  partitions of unity}\label{ssec.Sobolev} \emph{In this
  subsection, $M$ will be a manifold with boundary and bounded
  geometry, unless explicitly stated otherwise. This will be the case
  in most of the rest of the paper.}
  
We will need local descriptions of the Sobolev spaces using partitions
of unity. To this end, it will be useful to think of manifolds with
bounded geometry in terms of coordinate charts. This can be done by
introducing \emph{Fermi coordinates} on $M$ as in
\cite{Grosse.Schneider.2013}.  See especially Definition~4.3 of that
paper, whose notation we follow here. Recall that $\rinj(M)$ and
$\rinj(\pa M)$ denote, respectively, the injectivity radii of $M$ and
$\pa M$. Also, let $r_{\pa}\define \delta$ with $\delta$ as in
Definition~\ref{hyp_bdd_geo}.

Let $p \in \pa M$ and consider the diffeomorphism $\exp^{\pa
  M}_p\colon B^{T_{p}\pa M}_{r}(0) \to B^{\pa M}_r(p)$, if $r$ is
smaller than the injectivity radius of $\partial M$. Sometimes, we
shall identify $T_p \pa M$ with $\RR^{m-1}$ using an orthonormal
basis, thus obtaining a diffeomorphism $\exp^{\pa M}_p\colon
B^{m-1}_{r}(0) \to B^{\pa M}_r(p)$, where $B^{m-1}_{r}(0)\subset
\RR^{m-1}$ denotes the Euclidean ball with radius $r$ around $0\in
\RR^{m-1}$. Also, recall the definition of the normal exponential map
$\exp^{\perp} \colon \pa M \times [0, r_{\partial}) \to M$,
  $\exp^{\perp}(x, t) \define \exp_x^{M}(t\nu_x)$.  These two maps
  combine with the exponential $\exp_p^{M}$ to define maps
\begin{align}\label{eq.FC-chart}
\left\{\begin{matrix} \kappa_p \colon B^{m-1}_{r}(0) \times [0,r)\to
  M,\hfill & \kappa_p(x, t) \define \exp^{\perp}(\exp_{p}^{\pa M}(x),
  t),\hfill & \text{ if } p \in \pa M\\
 \kappa_p \colon B^{m}_{r}(0) \to M,\hfill & \kappa_p(v) \define
 \exp_p^{M}(v),\hfill & \text{ otherwise.}
\end{matrix}\right.
\end{align}
We let

\begin{align}\label{eq.Ugamma}
 W_p(r) \define
 \begin{cases}
  \kappa_p(B^{m-1}_{r}(0) \times [0,r)) \subset M & \text{ if } p \in
    \pa M\\
  \kappa_p(B^{m}_r(0)) = \exp_p^{M}(B^{m}_r(0)) & \text{ otherwise.}
 \end{cases}
\end{align}
In the next definition, we need to consider only the case $p \in \pa
M$, however, the other case will be useful when considering partitions
of unity.

\begin{definition}\label{FC-chart}
Let $p\in \partial M$ and $r_{FC} \define \min\left\{\, \frac{1}{2}
\rinj (\partial M), \, \frac{1}{4} \rinj (M),\, \frac{1}{2}
r_\partial\, \right\}$. Fix $0 < r \leq r_{FC}$. The map $\kappa_p
\colon B^{m-1}_{r}(0) \times [0,r) \to W_p(r)$ is called a \emph{Fermi
    coordinate chart} and the resulting coordinates $(x^{i}, r) \colon
  W_p(r) \to \RR^{m-1} \times [0, \infty)$ are called \emph{Fermi
      coordinates (around $p$)}.
\end{definition}

Figure~\ref{figFCchart} describes the Fermi coordinate chart.
\smallskip

\begin{figure}[ht!]
 \centering
 \begin{tikzpicture}[scale=0.8] 
\draw (0,0) -- (4,0); \draw[dotted] (0,0) -- (0,2) node[above]
      {$B^{m-1}(r)\times [0,r)$} -- (4,2) -- (4,0); \draw (1,1)
        node[right]{$(x,t)$} circle (0.02cm); \draw (2,0) node[below]
        {$0$} circle (0.02cm);

\draw[line width=1] (6,0) .. controls (9,1.5) .. (12,1) node[right]
     {$\partial M$}; \draw (9,1.24) node[below] {$p$} circle (0.03cm);
     \draw (8,0.95) circle (0.04cm); \draw[->, dashed] (8.5, 0)
     node[below] {$y=\exp^{\partial M}_p(x)$} -- (8,0.8);

\draw (8,0.9) .. controls (7.8,1.5) .. (8,2.2);

\draw[->] (8,0.95) -- (7.7,1.7); \draw (7.5, 1.5) node {$\nu_y $};
\draw (8,2.2) node[right] {$\exp^M_y(t\nu_y) $} circle (0.02cm);

\draw (6,0) .. controls (6.5,1.5) .. (6.5,3);
\draw (12,1) .. controls (12.5,2.5) .. (12.5,4) node[right] {$M$};
\draw (6.5,3) .. controls (9,4.5) .. (12.5,4);

\draw[dotted] (7,0.5) .. controls (6.8,1.5) .. (7.2,2.5);
\draw[dotted] (11,1.2) .. controls (10.8,2.5) .. (11.2,3.2)
node[xshift=13pt] {$W_p(r)$}; \draw[dotted] (7.2,2.5) .. controls (9,3.5)
.. (11.2,3.2);

\draw[->] (4.5,1) .. controls (5.3,1.5)  .. (6,1.3);
\draw (5.3,1.5) node[above] {$\kappa$};
\end{tikzpicture}
\caption{Fermi coordinates}\label{figFCchart}
\end{figure}
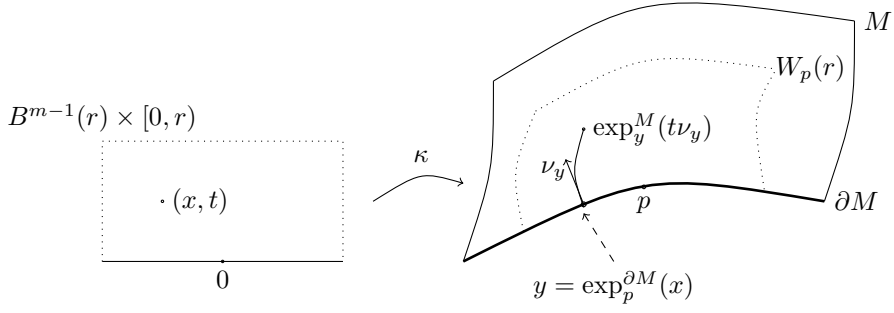
\smallskip

\begin{remark}\label{rem_FC_1}
Let $(M, g)$ be a manifold with boundary and bounded geometry.  Then
the coefficients of $g$ all their derivatives are uniformly bounded in
Fermi coordinates charts, see, for instance, \cite[Definition~3.7,
  Lemma~3.10 and Theorem~4.9]{Grosse.Schneider.2013}.
\end{remark}

For the sets in the covering that are away from the boundary, we will
use geodesic normal coordinates, whereas for the sets that intersect
the boundary, we will use Fermi coordinates as in
Definition~\ref{FC-chart}.  This works well for manifolds with bounded
geometry.  Note that
in this subsection, we \emph{do not} assume that $M$ has finite width.

Recall the notation of Equation~\eqref{eq.Ugamma}.
  
\begin{definition}\label{def.FC}  
Let $M^m$ be a manifold with boundary and bounded geometry. Assume as
in Definition~\ref{FC-chart} that $0 < r \leq r_{FC} \define
\min\left\{\, \frac{1}{2} \rinj (\partial M), \, \frac{1}{4} \rinj
(M),\, \frac{1}{2}r_\partial\, \right\}$. A subset
$\{p_\gamma\}_{\gamma \in I}$ is called an \emph{$r$-covering subset
  of $M$} if the following conditions are satisfied:
\begin{enumerate}[(i)]
\item For each $R>0$, there exists $N_R \in \NN$ such that, for each
  $p \in M$, the set $\{\gamma \in I\vert\, \dist(p_\gamma, p) <
  R\}$ has at most $N_R$ elements.
 \item For each $\gamma \in I$, we have either $p_\gamma \in \pa M$ or
   $d(p_\gamma, \pa M) \geq r$, so that $W_\gamma \define
   W_{p_\gamma}(r)$ is defined.
 \item $M \subset \cup_{\gamma = 1}^{\infty} W_{\gamma}$.
\end{enumerate}
\end{definition}

\begin{remark}\label{rem_FC} 
It follows as in \cite[Remark~4.6]{Grosse.Schneider.2013} that if $0 <
r < r_{FC}$ then an $r$-covering subset of $M$ always exists, since
$M$ is a manifold with boundary and bounded geometry. The picture
below shows an example of an $r$-covering set, where, the $p_\beta$'s
denote the points $p_\gamma \in \pa M$ and the $p_\alpha$'s denote the
rest of the points of $\{p_\gamma\}$.
\end{remark}

\begin{remark}
Let $(M, g)$ be a manifold with boundary and bounded geometry.  Let
$\{p_\gamma\}_{\gamma \in I}$ be an $r$-covering set and
$\{W_\gamma\}$ be the associated covering of $M$. It follows from (i)
of Definition \ref{def.FC} that the coverings $\{W_\gamma\}$ of $M$ and
$\{W_\gamma\cap \pa M\}$ of $\pa M$ are \emph{uniformly locally
  finite}, i.e.~there is an $N_0>0$ such that no point belongs to more
than $N_0$ of the sets $W_\gamma$.
\end{remark}

We shall need the following class of partitions of unity defined using
$r$-covering sets. Recall the definition of the sets $W_\gamma$ from
Definition~\ref{def.FC}(ii).

\begin{definition} \label{def.runif}
A partition of unity $\{\phi_\gamma\}_{\gamma \in I}$ of $M$ is
called \emph{an $r$-uniform partition of unity associated to the
  $r$-covering set $\{p_\gamma\} \subset M$ } (see
Definition~\ref{def.FC}) if
\begin{enumerate}[(i)]
 \item The support of each $\phi_\gamma$ is contained in $W_\gamma$.
 \item For each multi-index $\alpha$, there exists $C_\alpha > 0$ such
   that $|\pa^{\alpha}\phi_\gamma| \leq C_\alpha$ for all $\gamma$,
   where the derivatives $\pa^{\alpha}$ are computed in the normal
   geodesic coordinates, respectively Fermi coordinates, on
   $W_\gamma$.
\end{enumerate}
\end{definition}

\begin{center}
 
\usetikzlibrary{decorations.pathreplacing}
\begin{tikzpicture}[scale=1.1]

\draw[line width=1] (6,0) .. controls (9,1.5) .. (12,1) node[right]
     {$\partial M$}; \draw (6,0) .. controls (6.5,1.5) .. (6.5,3);
     \draw (12,1) .. controls (12.5,2.5) .. (12.5,4) node[right]
           {$M$}; \draw (6.5,3) .. controls (9,4.5) .. (12.5,4);

\draw[loosely dotted] (6.3,1) .. controls (9,2.5) .. (12.3,2);

\draw[fill] (9,1.24) circle (0.04cm); \draw[fill] (7.5,0.75) circle
(0.04cm); \draw[fill] (10.5,1.23) circle (0.04cm);

\draw[->, line width=0.2] (6.5, 4) node[above] {$p_\alpha$'s}
.. controls (7.5,3) ..  (8.9,2.7); \draw[->, line width=0.2] (6.5, 4)
.. controls (7.8,3.3) ..  (10.25,2.5); \draw[->, line width=0.2] (6.5,
4) .. controls (7,3) .. (7.7,2.25); \draw[->, line width=0.2] (6.5, 4)
.. controls (8,3.6) .. (9.75,3.2);

\draw[->, line width=0.2] (8.5, 0) node[below] {$p_\beta$'s}
.. controls (8,0.5) ..  (7.55,0.65); \draw[->, line width=0.2] (8.5,
0) -- (9,1.1); \draw[->, line width=0.2] (8.5, 0) .. controls (9,0.6)
.. (10.4,1.1);

\draw[dash pattern=on 1.5pt off 1.5pt on 1.5pt off 1.5pt] (8,0.95)
.. controls (8.25,1.5) .. (8.23,2) .. controls (9.2,2.35) .. (10.23,
2.28) .. controls (10.25, 1.6) .. (10, 1.25);
 
\draw[dash pattern=on 0.5pt off 1pt on 0.5pt off 1pt] (6.7,0.4)
.. controls (6.95,1) .. (6.93,1.34) .. controls (7.55,1.7) .. (8.6,
2.14) .. controls (8.65, 1.6) .. (8.4, 1.1);

\draw[dash pattern=on 0.5pt off 1pt on 0.5pt off 1pt] (9.7,1.4)
.. controls (9.95,2) .. (9.93,2.3) .. controls (10.65,2.25) .. (11.6,
2.1) .. controls (11.65, 1.6) .. (11.4, 1.1);

 \draw[dotted] (9,2.7) circle (1cm);
 \draw[fill] (9,2.7) circle (0.02cm);
 \draw[dotted] (10.3,2.5) circle (1cm);
 \draw[fill] (10.3,2.5) circle (0.02cm);
 \draw[dotted] (7.7,2.2) circle (1cm);
 \draw[fill] (7.7,2.2) circle (0.02cm);
 \draw[dotted] (9.8,3.2) circle (1cm);
 \draw[fill] (9.8,3.2) circle (0.02cm);
\draw[decoration={brace,mirror,raise=5pt},decorate]
  (6.1,1) -- node[left=6pt] {$\partial M\times [0,r)$} (6.1,0);

\end{tikzpicture}

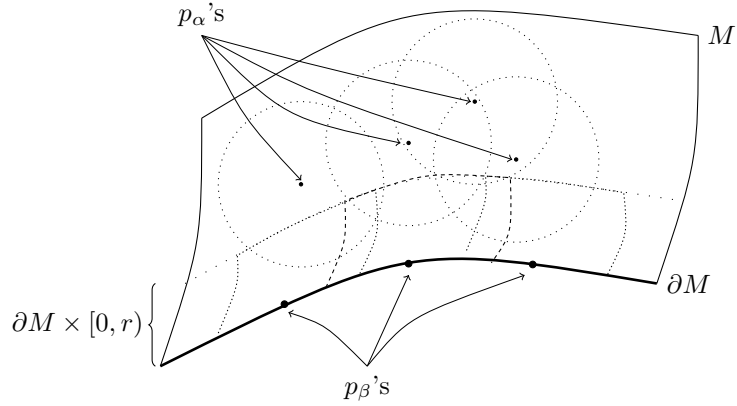
\captionof{figure}{A uniformly locally finite cover by Fermi and
  geodesic coordinate charts, compare with
  Remark~\ref{rem_FC}.}\label{figFCcover}
\end{center}

\begin{remark}
 Given an $r$-covering set $S$ with $r \leq r_{FC}/4$, an $r$-uniform
 partition of unity associated to $S \subset M$ always exists, since
 $M$ is a manifold with boundary and bounded geometry
 \cite[Lemma~4.8]{Grosse.Schneider.2013}.
\end{remark}

We have then the following proposition that is a consequence of
Remark~3.5 and Theorem~3.9 in \cite{Grosse.Schneider.2013}. See also
\cite{AmannAnis, AmannFunctSp, sobolev, BaerBallmann, BaerGinoux, 
  gerardBG, kordyukovLp2, Skrzypczak, TriebelBG} for related results,
in particular, for the use of partitions of unity.

\begin{proposition}\label{prop.part.unit}
Let $M^m$ be a manifold with boundary and bounded geometry.  Let
$\{\phi_\gamma\}$ be a uniform partition of unity associated to the
$r$-covering set $\{p_\gamma\} \subset M$ and let $\kappa_\gamma =
\kappa_{p_\gamma}$ be as in Equation~\ref{FC-chart}. Then
 \begin{align*}
  |||u|||^{2} \define \sum_{\gamma} \Vert (\phi_\gamma u) \circ
  \kappa_\gamma\Vert _{H^k}^{2}
 \end{align*}
 defines a norm equivalent to the usual norm $\Vert
 u\Vert_{H^{k}(M)}^2\define \sum_{i=0}^k \Vert \nabla^i
 u\Vert_{L^2(M)}^2$ on $H^{k}(M)$, $k\in \NN$.  Here $\Vert \cdot
 \Vert _{H^{k}}$ is the $H^{k}$ norm on either $\RR^{m}$ or on the
 half-space $\RR^{m}_+$.
\end{proposition}

Similarly, we have the following extension of the trace theorem to the
case of a manifold $M$ with boundary and bounded geometry, see
Theorem~5.14 in \cite{Grosse.Schneider.2013}. (See also \cite{sobolev}
for the case of Lie manifolds.)

\begin{theorem}[Trace theorem]\label{thm.trace}
 Let $M$ be a manifold with boundary and bounded geometry. Then, for
 every $s > 1/2$, the restriction to the boundary $\text{res}\colon
 \Gamma_c(M) \to \Gamma_c(\pa M)$ extends to a continuous, surjective
 map
 \begin{align*}
  \operatorname{res}\colon  H^s(M) \to H^{s-\frac{1}{2}}(\partial M).
 \end{align*}
\end{theorem}

\begin{notation}\label{not.notH-1}
We shall denote by $H^{1}_0(M)$ the closure of $\CIc(M \smallsetminus
\pa M)$. Then it coincides with the kernel of the trace map $H^1(M)
\to L^2(\pa M)$. We also have that $H^{-1}(M)$ identifies with the
(complex conjugate) dual of $H^1_0(M)$ with respect to the duality map
given by the $L^2$-scalar product. In general, we shall denote by
$H^1_D(M)$ the kernel of the restriction map $H^1(M) \to
L^2(\partialDM)$ and by $H^1_D(M)^*$ the (complex conjugate) dual of
$H^1_D(M)$ (see Equation \eqref{eq.def.H1D} for the definition of
these spaces).
\end{notation}

\section{The Poincar\'e inequality}
\label{sec.two}
Recall that the boundary of $M$ is partitioned, that is, that
we have fixed a boundary decomposition $\partial
M=\partialDM\amalg \partialNM$ into open and closed disjoint subsets
$\partialDM$ and $\partialNM$.  \emph{We assume in this 
  section that $(M, \partialDM)$ has \textbf{finite width}.} 
  In particular, in this section, $M$ is a manifold with
boundary and bounded geometry and $\partialDM$ has a non-empty
intersection with each connected component of~$M$.

\subsection{Uniform generalization of the
  Poincar\'e inequality}\label{ssec.unif} The following seemingly more
general statement is in fact equivalent to the Poincar\'e inequality
(Theorem \ref{thm_Poin_intro}).

\begin{theorem}\label{thm_Poin}
Let $M_\alpha$, $\alpha \in I$, be a family of $m$-dimensional smooth
Riemannian manifolds with partitioned smooth boundaries $\partial
M_\alpha = \partialDM_\alpha \amalg \partialNM_\alpha$.  Let us assume
that the disjoint union $M := \coprod M_\alpha$ is such that $(M,
\partialDM)$ has finite width, where $\partialDM := \coprod
\partialDM_\alpha$. Then, for every $p\in [1,\infty]$, there exists $0
< \cunif < \infty$ (independent of $\alpha$) such that
\begin{equation*} 
   \Vert f\Vert _{L^p(M_\alpha)} \leq \cunif \big(\Vert f\Vert
   _{L^p(\partialDM_\alpha)} + \Vert d f\Vert _{L^p(M_\alpha)} \big)
\end{equation*}
for all $\alpha \in I$ and all $f \in W^{1,p}_{\loc}(M_\alpha)$.
\end{theorem}

The main point of this reformulation of the Poincar\'e inequality is,
of course, that $\cunif$ is independent of $\alpha \in I$.  Here
$W^{1,p}$ refers to the norm $\Vert f\Vert_{W^{1,p}}^p = \Vert
f\Vert_{L^p}^p+\Vert df\Vert_{L^p}^p$ (with the usual modification for
$p=\infty$) and $W^{1,2}=H^1$. Moreover, $W^{1,p}_{loc}$,
respectively $L^p_{loc}$, consists of distributions that are locally in
$W^{1,p}$, respectively in $L^p$. Note that $f\in W^{1,p}_{\loc}(M)$ implies
$f|_{\pa M}\in L^p_{\loc}(\pa M)$ by standard local trace inequalities
\cite{Browder60}.  In particular, if $\Vert f\Vert _{L^p(\pa M)} +
\Vert d f\Vert _{L^p(M)}$ is finite, then $\Vert f\Vert _{L^p}$ is
also finite and the inequality holds.

Theorem~\ref{thm_Poin} follows from Theorem~\ref{thm_Poin_intro} with
the following arguments. Assume that a constant $\cunif>0$ as above
does not exist, then there is a sequence $\alpha_i\in I$ and a
sequence of non-vanishing $f_i\in W^{1,p}_{\loc}(M_{\alpha_i})$ with
\begin{equation*} 
  \frac{\big(\Vert f_i\Vert _{L^p(\partialDM_{\alpha_i})} + \Vert d
    f_i\Vert _{L^p(M_{\alpha_i})} \big)}{ \Vert f_i\Vert
    _{L^p(M_{\alpha_i})}}\to 0
\end{equation*}
By extending $f_i$ by zero to a function defined on all of $M$
(which is the {\em disjoint} union of all the manifolds
  $M_\alpha$), we would get a counterexample to
Theorem~\ref{thm_Poin_intro}.

Note that in the case that $I$ is uncountable, $M$ is not second
countable. Many textbooks do not allow such manifolds, see
Remark~\ref{def.mfd}. If one wants to restrict to second countable
manifolds, we can only allow countable sets $I$ above.

Using Kato's inequality, we obtain that the Poincar\'e inequality
immediately extends to sections in vector bundles equipped with
Riemannian or hermitian bundle metrics and compatible connections.  In
Example~\ref{examp.no}, we will give an example that
illustrates the necessity of the finite width assumption.

\subsection{Idea of the proof of the
  Poincar\'e inequality}\label{ssec.strat}

We shall now prove the version of the Poincar\'e inequality stated in
the Introduction (Theorem \ref{thm_Poin_intro}). We will do that for
all $p \in [1,\infty]$, although we shall use only the case $p = 2$ in
this paper. The proof of the Poincar\'e inequality will be split into
several steps and will be carried out in the next subsections.  For
the benefit of the reader, we first explain in
Remark~\ref{rem_sakurai} the idea of the proof of the Poincar\'e
inequality by Sakurai for the case $\pa M=\partialDM$.  In
Remark~\ref{rem_extend} we comment on the differences of that proof
and ours.

\begin{remark}[Idea of the proof for $\partial M=\partialDM$] 
\label{rem_sakurai} The case $\partial
M=\partialDM$ of Theorem \ref{thm_Poin_intro} was proven by Sakurai in
\cite[Lemma 7.3]{Sakurai} (for real valued functions and $p=1$) under
weaker assumptions than ours on the curvature tensor. We recall
the main ideas of Sakurai's proof: As in the classical case of bounded
domains, one starts with the one-dimensional Poincar\'e estimate on
geodesic rays emanating from $\partial M$ and perpendicular to
$\partial M$ up to the point where those geodesics no longer minimize
the distance to the boundary. The main point of the proof is that
every point is covered exactly once by one of these minimizing
geodesic perpendiculars, except for a zero measure set (the cut locus
of the boundary).  Another important point is that the constant in
each of these one-dimensional Poincar\'e estimates is given in terms
of the lengths of those minimal geodesics and, thus, is uniformly
bounded in terms of the width of the manifold $M$.  The global
Poincar\'e inequality then follows by integrating over $\partial M$
the individual one-dimensional Poincar\'e inequalities. Here one needs
to take care of the volume elements since $M$ is not just a product of
$\partial M$ with a finite interval -- but this can be done by
comparison arguments using the assumptions on the geometry.
\end{remark}

See also \cite{CarronHardy, HebeyNonlin, Saloff-CosteSobolev} for
Poincar\'e type inequalities on manifolds without boundary and bounds
on the Ricci tensor.

\begin{remark}[Idea of proof in the general case]\label{rem_extend} 
Now, we no longer require $\partialDM=\partial M$. In this case, we
use a similar strategy, with the difference that now we are using the
one-dimensional Poincar\'e inequality \emph{only} for geodesics
emanating perpendicularly from $\partialDM$. One can see then that an
additional technical problem occurs since there may be a subset of $M$
with non-zero measure that cannot be reached by such geodesics. This
phenomenon is illustrated in Figure~\ref{fig:ex_bound}. Nevertheless,
this problem can be circumvented by first assuming that the metric has
a product structure near the boundary.  The necessary preliminaries
for this part will be given in Section~\ref{ssec.prod} and the proof
of our Poincar\'e inequality in this case is carried out in
Subsection~\ref{ssec.poinprod}, following the method of
\cite{Sakurai}. The case of a general metric is then obtained by
comparing the given metric with a metric that has a product structure
near the boundary. We thus begin by examining the case of a product
metric (near the boundary). This is done in
Subsection~\ref{ssec.poingen} (next).  See also \cite{Kasue1, Kasue2}.
\end{remark}

\begin{figure}
\begin{tikzpicture}[scale=0.8] 
\draw[line width=1] (-3,0)--(4,0) node[right] {\small $\partialDM$};
\draw[line width=1] (-3,2) .. controls (4,2) and (-3,5) .. (0,5) --
(4,5) node[right] {\small $\partialNM$};

\begin{scope} 
\clip (-3,0) -- (-3,2) .. controls (4,2) and (-3,5) .. (0,5) -- (4,5)
-- (4,0) -- (-3,0); \draw[fill, opacity=0.05] (-3,0) -- (-3,5) --
(4,5)-- (4,0)--(-3,0);
\end{scope}

\begin{scope}
\clip (-3,0) -- (-3,2) .. controls (4,2) and (-3,5) .. (0,5) --
(0.24,5) --(0.24,2.8) -- (4,2.8) -- (4,0) -- (-3,0); \draw[fill,
  opacity=0.3, draw=none] (-3,0) -- (-3,2.8) -- (4,2.8)--
(4,0)--(-3,0);
\end{scope}

\draw[fill, opacity=0.3, draw=none] (0.24,2.8) -- (0.24,5) -- (4,5)--
(4,2.8)--(0.24,2.8); \draw[->] (3,0) node[below] {\small $x$} -- (3,1)
node[right] {\small $\nu_x$};
\end{tikzpicture}
\caption{Manifold $M\subset \RR^2$ with boundary $\partialDM \cup
  \partialNM$ where only a subset (in gray) can be reached by
  geodesics perpendicular to $\partialDM$.}\label{fig:ex_bound}
\end{figure}
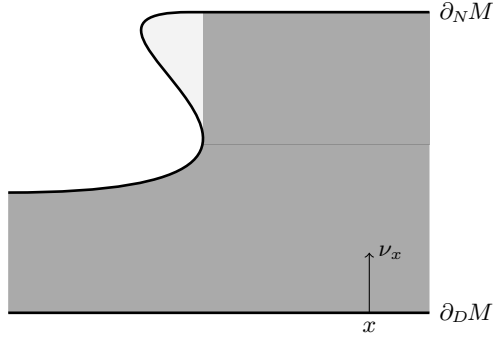

\subsection{Geometric preliminaries for
  manifolds with product metric} \label{ssec.prod} In the following
subsection, we will \emph{furthermore assume} that there
exists~$r_{\partial}>0$ such that the metric $g$ is a product metric
near the boundary, that is, it has the form
\begin{align}\label{eq.as.this.section}
 g = g_{\partial} + dt^2 , \text{\ on\ }
 \partial M\times [0, r_\partial ),
\end{align}
where $g_{\partial}$ is the metric induced by $g$ on $\partial M$.
 
We identify in what follows $\pa M \times \RR$ with the normal bundle
to $\pa M$ in $M$ as before for hypersurfaces, i.e.\ we identify
$(x,t)\in \pa M \times \RR$ with $t\nu_x$ in the normal bundle of
$\partial M$ in $M$.  Also, we shall identify $(x,t)\in \pa M\times
[0, r_{\pa})$ with $\exp^{\perp}(x, t)=\exp^{\perp}(t\nu_x) \in M$.

Our proof of the Poincar\'e inequality under the product metric
assumption on $\pa M \times [0, r_{\pa})$ is based on several
  intermediate results.  By decreasing $r_{\pa}$ we may assume that
  $\delta$ in Definition~\ref{hyp_bdd_geo}.(iv) and $r_{\pa}$ in
  Equation~\eqref{eq.as.this.section} are the same.

First, note that, by the product structure of $g$ near the boundary,
the submanifolds $\partial M\times \{t\}$ with $t\in [0,r_\partial)$
  are totally geodesic submanifolds of $M$ and that a geodesic in
  $\partial M\times [0,r_\partial)$ always has the form $c(t)=
    \exp^\perp(c_\partial (t), at)$ for some $a\in \mathbb R$ and some
    geodesic $c_\partial$ in $(\partial M, g_\partial)$.  This implies
    that a geodesic $c\colon [a,b]\to M$ with $c(a)\notin \partial M$
    and $c(b)\in \partial M$ cannot be extended to $[a,b+\epsilon]$
    for any $\epsilon>0$.

For a subset $A\subset M$ we define
\begin{align*}
  \dist(y,S) \define \inf_{x\in S} \{\dist(x,y)\}.
\end{align*}
Let $y \in M$. A \emph{shortest curve joining $y$ to $S \subset M$} is
by definition a rectifiable curve $\gamma\colon [a,b]\to M$ from $y$
to $S$ (that is, $\gamma(a) = y$, $\gamma(b) \in S$) such that no
other curve from $y$ to $S$ is shorter than $\gamma$. If a shortest
curve is parametrized proportional to arc length and its interior does
not intersect the boundary, then it is a geodesic. Such geodesics will
be called \emph{length minimizing geodesics}.

Remark \ref{rem_extend} shows that the following proposition is less
obvious than one might first think. In particular, the assumption that
we have a product metric near the boundary is crucial if $\pa_D M \neq
\pa M$.

\begin{proposition}\label{prop.a} We still assume that $(M, \partialDM)$ has 
finite width and the product structure assumption at the beginning of
Subsection~\ref{ssec.prod}. Then, for every $y\in M$, there is a
length minimizing smooth geodesic~$\gamma$ from~$y$ to~$\partialDM$.
\end{proposition} 

Note that obviously we need here our requirement that $\partialDM$
intersects all connected components of $M$, see
Definition~\ref{def.finite.width}.
Also note that, in general, there may be more than one (geometrically
distinct) shortest geodesic $\gamma\colon [a,b]\to M$ joining $y$ to
$\partialDM$.  If $y \notin \partial M$, every such curve satisfies
$\gamma([a,b))\cap \partial M=\emptyset$ and
  $\gamma'(b)\perp\partialDM$.

\begin{proof}
The proof is analogous to the classical proof of the Hopf-Rinow
Theorem (which implies that in a geodesically complete manifold, any
two points are joined by a length minimizing geodesic, see Chapter 7
in \cite{doCarmo}). We follow Theorem~2.8 (Hopf-Rinow) in Chapter~7 of
the aforementioned book. For a given point $y\in M\setminus\partialNM$
we define $r\define \dist (y,\partialDM)$. We only have to consider
the case $y\not\in \partialDM$, and in this case $r>0$. Let $\delta <
\rinj(M)$. It follows from the Gauss lemma that the length of any
curve joining $y$ to a point of the ``sphere'' $S_\delta(y) \define \{
\exp_y(\delta v) \vert\, |v|=1 \}$ is at least $\delta$, with the
infimum being attained by the image of the straight line under
$\exp_y$, see Chapter~3 in \cite{doCarmo} for details.  The function
$x\mapsto \dist (x,\partialDM)$ is continuous, and thus we can choose
a point $x_0\in S_\delta(y)$ with
\begin{align*}
  \dist (x_0,\partialDM)= \min\{\dist (x,\partialDM)\mid x \in
  S_\delta(y)\}.
\end{align*}
Every curve from $y$ to $\partialDM$ will intersect $S_\delta(y)$ somewhere,
thus we obtain
\begin{align}\label{eq.dist.min}
  \dist (y,x_0)+\dist (x_0,\partialDM)= \dist (y,\partialDM).
\end{align}
Let $|v| = 1$ with $\exp_y(\delta v) = x$.  We now claim that
\begin{align}\label{eq.claim}
 \exp_y(tv) \text{ is defined and }\dist (\exp_y(tv),\partialDM)=r-t
\end{align}
holds for all $t\in (0,r)$.  The proof is again analogous to the proof
of the theorem by Hopf-Rinow in \cite{doCarmo}. Let $A\define \{ t\in
[0,r]\mid \eqref{eq.claim}\text{ holds for }t\}$. Obviously $A$ is
closed in $[0,r]$, and from the triangle inequality we see that $t\in
A$ implies $\tilde t\in A$ for all $\tilde t\in [0,t]$. So $A=[0,b]$
for some $b\in[0,r]$. Further \eqref{eq.dist.min} implies $b\geq
\delta$. Due to the product structure near $\partial_NM$ the geodesic
$[0,b]\ni t\mapsto \exp_y(tv)$ does not hit $\partial_NM\times
      [0,d(y))$ for some $d(y)$ small enough, as otherwise this would
violate \eqref{eq.claim}.  We will show that for any $s_0\in A$,
$s_0<r$, there is a $\delta'>0$ with $s_0+\delta'\in A$. To this end,
we repeat the above argument for $y'\define \exp_y(s_0v)$ instead of
$y$. We obtain $\delta'>0$ and $x_0'\in S_{\delta'}(y')$ such that
$\dist (y',x_0')+\dist (x_0',\partialDM)= \dist (y',\partialDM)$, and
we write $x_0'=\exp_{y'}(\delta'v')$. This implies
\begin{align}\label{tripel.eq}
  \dist (y,y')+ \dist (y',x_0') + \dist (x_0',\partialDM) = \dist
  (y,\partialDM),
\end{align}
and then we get $\dist (y,y')+ \dist (y',x_0')= \dist (y,x_0')$. We
have shown that the curve
\begin{align*}
 \gamma(t) \define
  \begin{cases}
    \ \exp_y(tv) & \ 0 \leq t\leq s_0\\ \ \exp_{y'}((t-s_0)v') & \ s_0
    \leq t\leq s_0+\delta'\\
  \end{cases}
\end{align*}
is a shortest curve from $y$ to $x_0'$, and thus the geodesic is not
broken in $y'$. In other words
\begin{align*}
  \exp_y(tv)=\exp_{y'}((t-s_0)v') \quad s_0\leq t\leq
  s_0+\delta'.
\end{align*}
Using \eqref{tripel.eq} once again, we see that $s_0+\delta'\in A$.

We have seen that $b=\max A=r$. We obtain $\exp_y(rv) \in \partialDM$,
which gives the claim.  Moreover, the first variation formula implies
$\gamma'(r)\perp \partialDM$.
\end{proof}

\begin{proposition}\label{prop.fermi}
There is a continuous function $L\colon \partial_{D} M \to (0,\infty]$
such that the restriction of $\exp^\perp$ to
  \begin{align*}
    \{(x,t)\,|\ 0\leq t \leq L(x),\, x\in \partialDM \}
  \end{align*}
is surjective, and such that the restriction of $\exp^\perp$ to
  \begin{align*}
    \{(x,t)\,|\ 0< t<L(x),\, x\in \partialDM \}
  \end{align*}
is an embedding. Furthermore 
  \begin{align*}
    \dist (\exp^\perp(x,t), \partialDM) = t
  \end{align*}
if $0\leq t \leq L(x)$. The set
 \begin{align*}
    M_S \define \exp^\perp\left(\{(x,t)\,|\ t=L(x),\, x\in \partialDM
    \}\right)
 \end{align*}
is of measure zero.
\end{proposition}

\begin{proof}
For $x\in \partialDM$, let us consider the geodesic $\gamma_x\colon
I_x\subset [0,\infty)\to M$ with $\gamma_x(0)=x$ and
  $\gamma_x'(0)=\nu_x$, defined on its maximal
domain $I_x$. We choose $L(x)\in I_x $ as the maximal number such that
$\gamma_x|_{[0,L(x)]}$ realizes the minimal distance from $x$ to
$\gamma_x(t)$ if $0\leq t \leq L(x)$.  Let $y \in M$ and $d = \dist(y,
\partialDM)$. From the Proposition~\ref{prop.a} above, we see that a
shortest curve from $y$ to $\partialDM$ exists; in other words, there
is $x\in \partialDM$ with $y = \exp\sp{\perp}(x, d)$, where
$\exp\sp{\perp}(x, t) \define \exp_x(t \nu_x)$.  Therefore, the
restriction of $\exp^\perp$ to $\{(x,t)\,|\,0\leq t \leq L(x), x\in
\partialDM \}$ is surjective.  The continuity of $L$ is analogous to
\cite[Chapiter~13, Proposition~2.9]{doCarmo}.  As the geodesics
$t\mapsto \gamma_x(t)$ are minimizing for $0\leq t\leq L(x)$, it
follows similar to \cite[Chapiter~13, Proposition~2.2]{doCarmo} that
there is a unique shortest curve from $\gamma_x(t)$ to $\partialDM$ if
$0\leq t < L(x)$, and that the restriction of $\exp^\perp$ to
$\{(x,t)\,|\,0< t<L(x), x\in \partialDM \}$ is an injective
immersion. From the inverse function theorem we see that this
injective immersion is a homeomorphism onto its image, thus it is an
embedding.

The subset $M_S$ of $M$ is closed and has measure zero as it is the
image of the measure zero set $\{(x,L(x))\,|\, x\in \partialDM \}$
under the smooth map $\exp^\perp$.
\end{proof}

\begin{remark}
One usually defines the \emph{cut locus $\mathcal{C}(S)$ of a subset
  $S\subset M$} as the set of all points $x$ in the interior of $M$
for which there is a geodesic $\gamma\colon [-a, \epsilon)\to M$ with
  $\gamma(-a)\in S$, $\gamma(0)=x$, $\gamma$ being minimal for all
  $t\in (-a,0)$, but no longer minimal for $t>0$.  The name ``cut
  locus'' comes from the fact that this is the set where several
  shortest curves emanating from $S$ will either intersect classically
  or in an infinitesimal sense.  The relevance of this concept is that
  the set $M_S$ introduced in Proposition~\ref{prop.fermi} satisfies
  $M_S = \partialNM \cup \maC(\partialDM)$.
\end{remark}

If $H\colon T_pM\to T_qM$ is an endomorphism, then we express it in
an orthonormal basis as a matrix $A$. The quantity
\begin{align}\label{def.detH}
   |\det H| \define |\det A|
\end{align}
is then well-defined and does not depend on the choice of orthonormal
base.

For $x\in \partialDM$, $0\leq t\leq L(x)$, let $v(x,t)$, be the volume
distortion of the normal exponential map, that is,
\begin{align*}
 v(x,t) \define |\det d_{(x,t)}\exp^\perp|\,,
\end{align*}
with the absolute value of the determinant defined using a local
orthonormal base, as in \eqref{def.detH}. Clearly, $v(x, 0) = 1$.

\begin{proposition}\label{prop.distorsion}(Special
  case of \cite[Lemma~4.5]{Sakurai}) We assume that $M^{m}$ is a
  manifold with boundary and Ricci curvature bounded from
  below. Assume that the metric is a product near the boundary. We
  also assume that there exists $R>0$ with $\dist(x, \partialDM)<R$
  for all $x\in M$.  Then there is a constant $C>0$ such that, for all
  $x\in \partial M$ and all $0\leq s\leq t\leq L(x)$, we have
  \begin{align*}
  \frac{v(x,t)}{v(x,s)}\leq C.
  \end{align*}
The constant $C$ can be chosen to be $e^{(m-1)R\sqrt{|c|}}$, where
$(m-1)c$ is a lower bound for the Ricci curvature of $M$.
 \end{proposition}

This proposition is essentially a special case of the Heintze--Karcher
inequality \cite{karcher:89}.  We also refer to the section on
Heintze--Karcher inequalities in \cite{Ballmann:Riccati} for a proof
of the full statement, some historical notes, and some similar
inequalities.

\subsection{Proof in the case of a manifold with product
  boundary}\label{ssec.poinprod} We keep all the notations introduced
in the previous subsection, i.\thinspace e.\ Subsection
\ref{ssec.prod}. In particular, we assume that $g$ is a product metric
near the boundary.

\begin{proof}[Proof of Theorem~\ref{thm_Poin_intro}
    for $g$ a product near the boundary] We set
\begin{align*}
 \gamma_x(s) \define \exp^{\perp}(x,s) \define \exp_x^M(s\nu_x) \,,
\end{align*}
to simplify the notation. Recall that $\dvol_g$ denotes the volume
element on $M$ associated to the metric $g$
and $\dvol_{g \pa}$ denotes the corresponding volume form on
  the boundary. Let us assume first that
$p < \infty$. Since
\begin{align}\label{eq.bdry}
\int_M \ne{u}^{p} \dvol_g = \int_{\partialDM} \int_0^{L(x)}
\ne{u(\gamma_x(s))}^p\, v(x,s)\,ds \dvol_{g_\partial} \,,
\end{align}
and $\ne{\nabla{u}}\geq \ne{\nabla_{\gamma_x'(s)}{u}}$, it suffices to
find a $c > 0$ (independent on $x$) such that
\begin{align}\label{eqqqq} 
  \int_0^{L(x)} \ne{\nabla_{\gamma_x'(s)} u }^p \, v(x,s)\,ds \ \geq \ c
  \int_0^{L(x)} \ne{u(\gamma_x(t))}^p \, v(x,t)\,dt
\end{align}
for all $x\in \partialDM$. Indeed equations \eqref{eq.bdry} and
\eqref{eqqqq} yield Theorem~\ref{thm_Poin_intro} by integration over
$\partialDM$ (recall that for now $p < \infty$).

Let $f(s)\define u(\gamma_x(s))$, for some fixed $x\in M\setminus
(M_S\cup \partialDM)$. Then $f'(s) = \nabla_{\gamma_x'(s)} u$. Fix
$t\in [0,L(x)]$ and let $q$ be the exponent conjugate to $p$, that is,
$p^{-1} + q^{-1} = 1$. Using $f(t)=f(0) + \int_0^t f'(s) ds$, we
obtain, for $p < \infty$,
\begin{align*} 
   |f(t)|^p v(x,t)& \leq \left(|f(0)| + \int_0^t |f'(s)|\, ds\right)^p
   v(x,t)\\
& \leq 2^{p-1}\left[\, |f(0)|^p + \left ( \int_0^t |f'(s)|\,
     ds\right)^p \, \right ] v(x,t) \\
 & \leq 2^{p-1}|f(0)|^p v(x,t) + 2^{p-1}t^{p/q} \int_0^t |f'(s)|^p\,
   v(x,t)\,ds\\
  &\leq C |f(0)|^p + C L(x)^{p/q} \int_0^{L(x)} |f'(s)|^p\, v(x,s)\,ds
   \\
 & \leq C |f(0)|^p + \, C R^{p/q} \int_0^{L(x)} |
   \nabla_{\gamma_x'(s)} u |^p \, v(x,s)\,ds.
\end{align*}
We have used here the fact that $v(x, 0) = 1$ and, several times,
Proposition \ref{prop.distorsion}.  Hence, integrating once more with
respect to $t$ from $0$ to $L(x) \le R$, we obtain
\begin{equation*} 
  \int_0^{L(x)} |f(t)|^p v(x,t) dt \leq C R |f(0)|^p + C R^{p}
  \int_0^{L(x)} |\nabla_{\gamma_x'(s)} u |^p v(x,s)ds ,
\end{equation*}
which gives \eqref{eqqqq} by integration with respect to
$\dvol_{g_\partial}$,  and
hence our result for $p < \infty$. The case $p = \infty$ is
simpler. Indeed, it suffices to use instead
\begin{equation*} 
  |f(t)| \leq |f(0)| + \int_0^t |f'(s)| ds \leq |f(0)| + t \Vert
  f'\Vert _{L^\infty}
 \leq |f(0)| + R \Vert d u\Vert _{L^\infty(M)}
\end{equation*}
By taking the `sup' with respect to $t$ and $x \in \partialDM$, we
obtain the result.
\end{proof}

\subsection{The general case}\label{ssec.poingen}
We now show how the general case of the Poincar\'{e} inequality for a
metric with bounded geometry and finite width on $M$ can be reduced to
the case when the metric is a product metric in a small tubular
neighborhood of the boundary, in which case we have already proved the
Poincar\'e inequality.

The general case of Theorem~\ref{thm_Poin_intro} follows directly from
the special case where~$g$ is product near the boundary and the
following lemma.  Recall that we identify $\pa M \times [0,
  r_{\partial})$ with its image in $M$ via the normal exponential map
  $\exp^{\perp}$.

\begin{lemma}\label{lem_prod_bound} 
Let $(M^{m},g)$ be a manifold with boundary and bounded geometry. Let
$g_\partial$ be the induced metric on the boundary $\partial M$. Then
there is a metric $g'$ on $M$ of bounded geometry such that
\begin{enumerate}
\item $g'=g_\partial+ dt^2$ on $\partial M\times [0, r']$ for some
  $r'\in (0, r_\partial)$
\item $g$ and $g'$ are equivalent, that is, there is $C>0$ such that
  $C^{-1}g \leq g' \leq Cg$.
\end{enumerate}
In particular, the norms $| \cdot |_g$ and $| \cdot |_{g'}$ on
$E$-valued one-forms, respectively on the volume forms for $g$ and
$g'$, are equivalent.
\end{lemma}

\begin{proof}[Proof of Lemma~\ref{lem_prod_bound}]
Let $0 < r' < r_{FC}/4 \define \frac14 \min\left\{\, \frac{1}{2} \rinj
(\partial M), \, \frac{1}{4} \rinj (M),\, \frac{1}{2} r_\partial\,
\right\}$ be small enough, to be specified later. Here $r_{FC}$ is as in
the choice of our Fermi coordinates in Definition~\ref{FC-chart}.  Let
$\eta\colon [0,3r']\to [0,1]$ be a smooth function with
$\eta|_{[0,r']}=0$ and $\eta|_{[2r',3r']}=1$. We set $g'(x,t)\define
\eta(t) g(x,t) + (1-\eta(t))(g_\partial (x) + dt^2)$ for $(x,t)\in
\partial M\times [0, 3r']$ and $g'=g$ outside $\partial M\times [0,
  3r']$. By construction $g'$ is smooth and is a product metric on
$\pa M \times [0, r']$.  The rest of the proof is based on the use of
Fermi coordinates around any $p \in \pa M$ to prove that $g$ and $g'$
are equivalent for $r'$ small enough.

Then, in the Fermi coordinates $(x, t) \define \kappa_p^{-1}$ around
$p\in \partial M$, see Equation~\eqref{eq.FC-chart}, we have
$g_{ij}(x,t) = g_{ij}(x,0)+tg_{ij,t}(x,0)+O(t^2)$, $g_{it}(x,0)=0$ for
$i\neq t$, $g_{tt}(x,t)=1$, and $(g_\partial)_{ij}(x)= g_{ij}(x,0)$
for $i,j\neq t$. Thus,
\begin{align}\label{gij}
 g'_{ij}(x,t) - g_{ij}(x,t) \, = \
 \begin{cases}
  -(1-\eta(t))(tg_{ij,t}(x,0)+O(t^2)) \ & \ \text{ if } (i, j) \neq
  (t, t)\\
  \ \ \ 0 \ \ \ \ & \ \text{ otherwise}.
 \end{cases}
\end{align}
Since $g_{ij,t}$ is uniformly bounded by Remark~\ref{rem_FC_1}, we
obtain $|g'_{ij}(x,t)-g_{ij}(x,t)|\leq tC$, for $(x,t)\in
B^{m-1}_{r'}\times [0,r')$, where the constant $C$ is independent of
the chosen~$p$. We note that, in these coordinates, the metric $g$ is
equivalent to the Euclidean metric on $B^{m-1}_{r'}(0) \times [0, r')
  \subset \RR^{m}$ in such a way that the equivalence constants do not
  depend on the chosen $p$. This can be seen from
\begin{align*}
 \big |g_{ij}(x,t)-g_{ij}(0,0) \big | & = |g_{ij}(x,t)-\delta_{ij}|
 \\
 &\leq \sup | \nabla_{(x,t)} g_{ij}(x,t)| \Vert (x,t)\Vert +O(t^2)
 \leq Cr'+O(t^2),
\end{align*}
where $C$ is the uniform bound for $\nabla_{(x,t)} g_{ij}(x,t)$, which
is finite by the bounded geometry assumption. Moreover, $O(t^2)\leq
ct^2$ with $c$ depending on the uniform bound of the second
derivatives on $g_{ij}(x,t)$ and $r_{FC}$.  Let $X$ be a vector in
$(x,t)$. Then
\begin{align*}
 \big |g(X,X)-|X|^2 \big | = \Big| \, \sum_{ij}
 (g_{ij}-\delta_{ij}) X^iX^j \, \Big | \ \leq \ Cr' |X|^2 \,.
\end{align*}
Thus, for $r'$ such that $Cr'<1$, it follows that $g$ and the
Euclidean metric are equivalent on the chart around $p$ in such a way
that the constants do not depend on $p$.  Similarly, we then obtain
\begin{align*}
 \big| g'(X,X)-g(X,X) \big | \ & = \ \Big| \, \sum_{ij}
 (g'_{ij}(x,t)-g_{ij}(x,t)) X^iX^j \, \Big | \\
 & \leq \ r' C |X|^2\leq r'C (1-Cr')^{-1} g(X,X)\,.
\end{align*}
Hence, $g'$ and $g$ are equivalent for $r$ small.

In particular, $|\det g_{ij}(x,t)-\det g'_{ij}(x,t)|\leq c t$ for a
positive $c$ independent of $x$, $p$, and $t$. Since $M$ has bounded
geometry, $\det g_{ij}$ is uniformly bounded on all of $M$ both from
above and away from zero. Therefore the volume forms for $g$ and $g'$
are equivalent.

An estimate similar to \eqref{gij} holds for $ (g')^{ij}(x,t) -
g^{ij}(x,t)$. Together with the relation $|\alpha|_g^2(p)=\sum_{i,j}
g^{ij}(p) \alpha_p(e_i)\alpha_p(e_j)$ for a one-form $\alpha$, this
gives the claimed result for the one-forms
\end{proof}

\begin{remark}
The proof implies that the constant $c$ in the Poincar\'e inequality
can be chosen to only depend on the bounds for~$R^M$ and its
derivatives, on the bounds for~$\II$, on~$r_{FC}$
(Definition~\ref{FC-chart}), on $p \in [1, \infty]$, and on the width
of $(M, \pa_D M)$. See also \cite{Sakurai}.
\end{remark}

With Lemma \ref{lem_prod_bound}, the proof of Theorem
\ref{thm_Poin_intro} is now complete.

\section{Invertibility of the Laplace operator}
\label{sec.three}

We now proceed to apply our Poincar\'e inequality to the study of the
spectrum of the Laplace operator with mixed boundary conditions and to
its invertibility in the standard scale of Sobolev spaces. In
  Subsection \ref{ssec.4.1} $M$ will be an arbitrary Riemannian
  manifold with boundary. However, in Subsections \ref{ssec.4.2} and
  \ref{ssec.4.3}, we will resume our assumption that $M$ has bounded
  geometry.

\subsection{Lax--Milgram, Poincar\'e, and well-posed\-ness in 
energy spaces\label{ssec.LaxMilgram}} 
\label{ssec.4.1}
In this subsection, we discuss the relation between the Poincar\'e
inequality and well-posed\-ness in $H^{1}$ for the Poisson problem
with suitable mixed boundary conditions. {\em We {\em do not} assume,
  in this subsection, that $M$ has bounded geometry, except when
  stated otherwise.} The bounded geometry assumption will be needed,
however, for one of the main results of this paper, which is the
well-posed\-ness of the Poisson problem with suitable mixed boundary
conditions on manifolds with boundary and finite width, see
Definition~\ref{def.finite.width} and Theorem~\ref{thm.H1new} below.
We define the semi-norms $|u|_{W^{1,p}(M)} \define \Vert
du\Vert_{L^p(M;T^*\!M)}$ and $|u|_{H^1(M)}\define |u|_{W^{1,2}(M)}$.

\begin{definition}\label{def.pPoincare}
We say that $(M, \partialDM)$ \emph{satisfies the $L^p$-Poincar\'e
  inequality} if there exists $c_p > 0$ such that
\begin{align*}
 \Vert u\Vert _{L^p(M)} \leq c_p \Vert d u\Vert _{L^p(M;T^*M )} =: c_p
 |u|_{W^{1, p}(M)},
\end{align*}
for all $u \in H^1_D(M)$ (recall Notations~\ref{not.notH-1}). If $(M,
\partialDM)$ satisfies the $L^2$-Poincar\'e inequality, we define
$c_{M, \partialDM}$ to be the least $c_2$ with this property, compare
\eqref{eq.def.cmdm}. Otherwise we set $c_{M, \partialDM} = \infty$.
\end{definition}

We have the following simple lemma.

\begin{lemma}\label{lemma.eq.norms} 
The semi-norm $|\, \cdot \,|_{H^1(M)}$ is equivalent to the $H^1$-norm
on $H^1_{D}(M)$ if, and only if, $(M, \partialDM)$ satisfies the
$L^2$-Poincar\'e inequality.  In particular, this is true if $(M,
\partialDM)$ has finite width.
\end{lemma}

\begin{proof} 
For the simplicity of the notation, we omit below the manifold~$M$
from the notation of the (semi-)norms.  Let us assume that $(M,
\partialDM)$ satisfies the $L^2$-Poincar\'e inequality. By definition
we have $|u|_{H^1} \leq \Vert u\Vert_{H^1}$, so to prove the
equivalence of the norms, it is enough to show that there exists $C >
0$ such that $C |u|_{H^1} \geq\Vert u\Vert _{H^1}$ for all $u \in
H^1_{D}$. Indeed, the $L^2$-Poincar\'{e} inequality gives
 \begin{align}\label{eq.equiv.norm}
  \underbrace{(c_{M,\partialDM}^2+1)}_{C:=} |u|_{H^1}^2 \geq \Vert
  u\Vert _{L^{2}}^2 + |u|_{H^1}^2 =: \Vert u\Vert _{H^1}^2.
 \end{align}
Conversely, if the two norms are equivalent, then we have for $u\in
H^1_D(M)$
 \begin{align*}
  \Vert u\Vert _{L^{2}} \leq \Vert u\Vert _{H^1} \leq C |u|_{H^1}
  \define C\Vert d u\Vert _{L^2}.
 \end{align*}
The proof is complete.
\end{proof}

Clearly, the last lemma holds for functions in $W^{1, p}$, for
  all $p \in [1, \infty]$.
We shall need the Lax--Milgram lemma. Let us recall first
the following well-known definition.

\begin{definition}\label{def.strongly.coercive}
Let $V$ be a Hilbert space and let $P \colon V \to V^{*}$ be a bounded
operator. We say that $P$ is \emph{strongly coercive} if there exists
$\gamma > 0$ such that $| \<Pu, u\>| \geq\gamma \Vert u\Vert _V^2$ for
all $u \in V$. The ``best'' $\gamma$ with this property (i.e. the
largest) will be denoted $\gamma_P$.
\end{definition}

\begin{lemma}[Lax--Milgram lemma]\label{lemma.LaxMilgram}
Let $V$ be a Hilbert spaces and $P \colon V \to V^{*}$ be a strongly
coercive map. Then $P$ is invertible and $\Vert P^{-1}\Vert \leq
\gamma_P^{-1}$.
\end{lemma}

For a proof of the Lax-Milgram lemma, see, for example, \cite[Section
  5.8]{TrudingerBook} or \cite{Nirenberg55, JostBook}.  The following
result explains the relation between the Poincar\'e inequality and the
Laplace operator. Note that $H^1_D(M)\subset L^2(M)\subset
H^1_D(M)^*$.  In particular, $\Delta-\lambda\colon H^1_D(M) \to
H^1_D(M)^*$, for $\lambda\in \CC$, is \emph{defined} by $\< (\Delta
-\lambda) u, v\> = (du, dv)-\lambda (u,v)$, where $\< \, , \, \>$ is
the duality pairing (considered conjugate linear in the second
variable).

Recall the constant $c_{M, \partialDM}$ of
Definition~\ref{def.pPoincare}. In particular, $(1 + c_{M, \pa
  M}^2)^{-1} = 0$ precisely when $(M, \partialDM)$ does {\em not}
satisfy the $L^2$-Poincar\'e inequality.  Also, recall the definition
of the spaces $H^1_D$ from Equation \eqref{eq.def.H1D}.

\begin{proposition}\label{prop.isomorphism} 
Assume that $M$ is a manifold with boundary (no bounded geometry
assumption).  We have that the map $\Delta \colon H^1_{D}(M) \to
H^1_{D}(M)^{*}$ is an isomorphism if, and only if, $(M, \partialDM)$
satisfies the $L^2$-Poincar\'e inequality.  Moreover, if $\Re(\lambda)
< (1 + c_{M, \pa M}^2)^{-1}$, then $\Delta -\lambda\colon H^1_{D}(M)
\to H^1_{D}(M)^{*}$ is strongly coercive and hence an isomorphism.
\end{proposition}

\begin{proof} 
Indeed, let us assume that $\Delta \colon H^1_{D}(M) \to H^{1}_{D}(M
)^{*}$ is an isomorphism. If the $L^2$-Poincar\'e inequality was not
satisfied, then there would exist a sequence of functions $u_n \in
H^1_D(M)$ such that $\Vert u_n\Vert _{H^1} = 1$, but $\Vert d u_n\Vert
_{L^2}^2 = \<\Delta u_n, u_n\> \to 0$ as $n \to \infty$, by
Lemma~\ref{lemma.eq.norms}. Since $\Delta$ is an isomorphism, there is
a $\gamma > 0$ such that $\Vert \Delta u\Vert _{(H^1_D)^*} \geq2\gamma
\Vert u\Vert _{H^1}$.  Therefore, using the definition of the (dual)
norm on $H^1_D(M)^*$ we can find a sequence $v_n \in H^1_D(M)$ such
that $\Vert v_n \Vert _{H^1(M)} = 1$ and $\<\Delta u_n , v_n\>
\geq\gamma$. Using the Cauchy-Schwarz inequality for the $L^2$-scalar
product on $1$-forms, we then obtain the following
 \begin{align*}
  \Vert d u_n\Vert _{L^2}^2 = \Vert d u_n\Vert_{L^2}^2 \Vert v_n\Vert
  _{H^1}^2 & \geq \Vert d u_n\Vert_{L^2}^2 \Vert d v_n\Vert_{L^2}^2\\
   &\geq |(d u_n, d v_n)|^2 =|\< \Delta u_n, v_n\>|^2 \geq\gamma^2 >
  0,
 \end{align*}
which contradicts $\Vert d u_n\Vert _{L^2} \to 0$.

Conversely, let $\Re(\lambda) < (1 + c_{M, \pa M}^2)^{-1}$ and $u \in
H^1_D(M)$. Using \eqref{eq.equiv.norm} we have
\begin{multline*} 
 \Re\left( \< (\Delta - \lambda) u, u \>\right) = (du, du) -
 \Re(\lambda) (u, u) \seq |u|_{H^1}^2 - \Re(\lambda) \|u\|_{L^2}\\
 \geq \left ( {(1 + c_{M, \partialDM}^2)^{-1}} - \Re(\lambda) \right )
 \Vert u\Vert _{H^1}^2.
\end{multline*}
The operator $\Delta - \lambda \colon H^1_D(M) \to H^1_D(M)^*$ thus
satisfies the assumptions of the Lax--Milgram lemma, and hence $\Delta
- \lambda$ is an isomorphism for $\Re(\lambda) < (1 + c_{M, \pa
  M}^2)^{-1}$.
\end{proof}

In our setting this directly gives

\begin{theorem}\label{thm.H1new} 
If $(M, \partialDM)$ has finite width, then $(1 + c_{M, \pa M}^2)^{-1}
> 0$, and hence $\Delta \colon H^1_{D}(M) \to H^1_{D}(M)^{*}$ is an
isomorphism.
\end{theorem}

The converse of this result is not true, the following two examples
show that, without the assumption of finite width, $\Delta \colon
H^1_{D}(M) \to H^1_{D}(M)^{*}$ may fail to be an isomorphism.

\begin{example}
Let $M$ be the $n$-dimensional hyperbolic space $\HH^n$, $n\geq 2$,
whose boundary is empty. Hence $(\HH^n,\emptyset)$ does {\em not} have
finite width, but the $L^2$-Poincar\'e inequality holds with
  $$c_{(\HH^n.\emptyset)}=2/(n-1)$$ since $\frac{(n-1)^2}{4}$ is the
infimum of the $L^2$-spectrum of the Laplacian on $\HH^n$,
\cite{Donnelly, Helgason}. Thus $\Delta\colon H^1(\HH^n) \to
H^1_D(\HH^n)^* = H^{-1}(\HH^n)$ is an isomorphism. So the finite width
condition is not necessary for the Laplacian on a manifold with
boundary and bounded geometry to be invertible. See also
  \cite{CarronPedon, Pasquale17, Salah2010} for further results on the
  spectrum of the Laplacian on symmetric spaces.
\end{example}

\begin{example} \label{examp.no}
Let us assume that in Example \ref{ex.Omega} $M_0 = \RR^{m-1}$ with
the standard euclidean metric. As in that example, $f > g$ are smooth
functions such that $df$ and $dg$ are totally bounded (i.e.  bounded
and with all their covariant derivatives bounded as well). Assume that
the functions $f$ and $g$ satisfy also that $\lim_{|x|\to \infty} f(x)
- g(x) = \infty$.  Then, using the notation of Example \ref{ex.Omega},
we have that $M\define\Omega(f, g)$ is a manifold with boundary and
bounded geometry. However, $(M, \pa M)$ does not have finite
width. Moreover, it does not satisfy the Poincar\'e inequality. This
can easily be seen as follows: Since $0$ is in the essential spectrum
of $\Delta$ on $\mathbb R^n$, there is sequence $v_i\in
C_c^\infty(B_i(0))$ with $\Vert v_i\Vert_{L^2}=1$ and $\Vert
dv_i\Vert_{L^2}\to 0$.  By translation we can assume that the support
of $v_i$ is in $\Omega(f,g)$. Thus, the Poincar\'e inequality is
violated and by Proposition~\ref{prop.isomorphism}, the operator
$\Delta$ (with any kind of mixed boundary conditions) is not
invertible.
\end{example}

\subsection{Higher regularity}\label{ssec.4.2}
We consider in this subsection the regularity and invertibility of the
Laplacian in the scale of Sobolev spaces determined by the metric,
which is the main question considered in this paper (see the
Introduction). We assume that $(M, \partialDM)$ is a manifold with
boundary and bounded geometry, however, for regularity questions {\em
  we do not require $(M, \partialDM)$ to have finite width}.

We shall use the notation introduced in Equations \eqref{eq.FC-chart}
and \eqref{eq.Ugamma} and in Definition~\ref{def.FC}. We set
$P_x\define \kappa_x^* \circ \Delta \circ (\kappa_x)_*$, where
$\kappa_x$ is viewed as a diffeomorphism on its image. Thus $P_x$ is a
differential operator on the Euclidean ball, respectively cylinder,
corresponding to the geodesic normal coordinates, respectively Fermi
coordinates, on $W_x\define W_x(r)$, see \eqref{eq.Ugamma}.  Thus
$\<P_x u, v\> = \int_{W_x} (du, dv) \dvol_g$.  We endow the space of
these differential operators with the norm defined by the maximum of
the $W^{2,\infty}$-norms of the coefficients.

\begin{lemma} \label{lemma.bd.fam}
Let $r < r_{FC}$. Then the set $\{P_x \vert \ \dist(x, \pa M) \geq r
\}$ is a relatively compact subset of the set of differential
operators on the ball $B_r^m(0) \subset \RR^{m}$. Similarly, the set
$\{P_y\vert \ y \in \pa M \}$ is a relatively compact subset of the
set of second order differential operators on $b(r)\define
B^{m-1}_r(0)\times [0,r)\subset \mathbb R^{m}$.
\end{lemma}

\begin{proof}
The coefficients of the operators $P_x$ and all their derivatives are
uniformly bounded, by assumption. By decreasing $r$, if necessary, we
get that they are bounded on a compact set.  The Arzela-Ascoli theorem
then yields the result.
\end{proof}

We get the following result 

\begin{lemma} \label{lemma.bd.fam2}
For any $k \in \NN$ (so $k \geq1$), there exists $C_k>0$ such that
\begin{align*}
 \Vert w \Vert _{H^{k+1}(M)} \leq C_k \big( \Vert \Delta w \Vert
 _{H^{k-1}(M)} + \Vert w \Vert _{H^{1}(M)} + \Vert w \Vert
 _{H^{k+1/2}(\pa M)} \big),
\end{align*}
for any $w \in H^{1}(W_\gamma)$ with compact support, where $W_\gamma$
is a coordinate patch of Definition~\ref{def.FC}.
\end{lemma}

Any undefined term on the right hand side of the inequality in Lemma
\ref{lemma.bd.fam2} is set to be $\infty$ (for instance $\Vert w \Vert
_{H^{k+1/2}(\pa M)} = \infty$ if $w \notin H^{k+1/2}(\pa M)$), in
which case the stated inequality is trivially satisfied.  Note also
that the condition that $w$ have compact support in $W_\gamma$ does
{\em not} mean that $w$ vanishes on $W_\gamma \cap \pa M$. (This
comment is relevant, of course, only if $W_\gamma \cap \pa M \neq
\emptyset$.)

\begin{proof}
Elliptic regularity for strongly elliptic equations (see, for
instance, Theorem~8.13 in \cite{TrudingerBook}, Theorem~9.3.3 in
\cite{JostBook}, or Proposition~11.10 in \cite{Taylor1}) gives, for
every $\gamma\in I$ as in Definition~\ref{def.FC}, that there exists
$C_\gamma$ such that
\begin{align*}
 \Vert w \Vert _{H^{k+1}(M)} \leq C_\gamma \big( \Vert \Delta w\Vert
 _{H^{k-1}(M)} + \Vert w \Vert _{H^{1}(M)} + \Vert w \Vert
 _{H^{k+1/2}(\pa M)} \big)
\end{align*}
for any $w \in H^{1}(W_\gamma)$ with compact support.  Of course, if
$W_\gamma$ is an interior set (that is, it does not intersect the
boundary), then $\Vert w \Vert _{H^{k+1/2}(\pa M)} = 0$. This should
be understood in the sense that if the right hand side is finite, then
$\Vert w \Vert _{H^{k+1}(M)} < \infty$, and hence $w \in
H^{k+1}(M)$. We need to show that we can choose $C_\gamma$
independently of $\gamma$. Let us assume the contrary. Then, for a
suitable subsequence $p_j$, we have that there exist $w_j \in
H^{1}(W_j)$ with compact support such that
\begin{align*}
 \Vert w_j \Vert _{H^{k+1}(M)} > 2^{j} \big( \Vert \Delta w_j \Vert
 _{H^{k-1}(M)} + \Vert w_j \Vert _{H^{1}(M)} + \Vert w_j
 \Vert_{H^{k+1/2}(\pa M)} \big).
\end{align*}
We can assume that the points $p_{j}$ are either all at distance at
least $r$ to the boundary, or that they are all on the boundary.  Let
us assume that they are all on the boundary. The other case is very
similar (even simpler). Using Fermi coordinates $\kappa_j\colon b(r)
:= B_r^{m-1}(0) \times [0, r) \to W_j$ we can move to a fixed
  cylinder, but with the price of replacing $\Delta$ with $P_j \define
  P_{p_j}$, a variable coefficient differential operator on
    $b(r)$.  By Lemma~\ref{lemma.bd.fam}, after passing to another
  subsequence, we can assume that the coefficients of the
  corresponding operators $P_j$ converge in the $\maC^{\infty}$
  topology on $b(r)$ to the coefficients of a fixed operator
  $P_\infty$.  In particular, we can assume that
\begin{align}\label{eq.epsilon}
 \Vert P_j u\Vert _{H^{k-1}(b(r))} \geq \Vert P_\infty u\Vert
 _{H^{k-1}(b(r))} - \epsilon_j \Vert u\Vert _{H^{k+1}(b(r))},
\end{align}
for all $u \in H^{k+1}(b(r))$, where $\epsilon_j \to 0$ as $j \to
\infty$ independent of $w$.  

The main point of this proof (exploited in greater generality in
\cite{GN17}) is that $P_\infty$ also satisfies elliptic
regularity. This is because it is strongly elliptic, since this
property is preserved by limits in which the strong ellipticity
constant is bounded from below {\em and} strongly elliptic operators
with Dirichlet boundary conditions satisfy elliptic regularity (see
the references above).  This means that there exists $C > 0$ such that
\begin{equation*}
 C \big ( \Vert P_\infty u\Vert _{H^{k-1}(b(r))} + \Vert u \Vert
 _{H^{1}(b(r))} + \Vert u \Vert _{H^{k+1/2}(B_r^{m-1}(0))} \big ) \geq
 \Vert u \Vert _{H^{k+1}(b(r))}
\end{equation*}
for all $u \in H^{k+1}(b(r))$ with compact support in $b(r) :=
B_r^{m-1}(0) \times [0, r)$.
 
By Proposition~\ref{prop.part.unit} we also have

\begin{align}\label{eq.equiv.norms}
 \frac1{1+C_0} \, \leq \, \frac{\Vert w\circ \kappa_j^{-1}\Vert
   _{H^{\ell}(W_j)}}{\Vert w\Vert _{H^{\ell}(b(r))}} \, \leq \,
 1+C_0\,,
\end{align}
for some $C_0 > 0$, for all $w \in H^{k+1}(b(r))$, and for $\ell \leq
k+1$, by the bounded geometry of $M$.

Equations \eqref{eq.epsilon} and \eqref{eq.equiv.norms} then give,
first in the boundary-less case
\begin{align*}
  \Vert w_j\circ \kappa_j \Vert _{H^{k+1}(b(r))} & \leq C \big (\,
  \Vert P_\infty (w_j\circ \kappa_j)\Vert _{H^{k-1}(b(r))} + \Vert
  w_j\circ \kappa_j \Vert _{H^{1}(b(r))} \, \big ) \\
 & \leq C \big (\, \Vert P_j(w_j\circ \kappa_j)\Vert _{H^{k-1}(b(r))}
  + \Vert w_j\circ \kappa_j \Vert _{H^{1}(b(r))}\\
 & \phantom{\leq} + \epsilon_j \Vert w_j\circ \kappa_j \Vert
  _{H^{k+1}(b(r))}\, \big ) \\
 & \leq C (1+C_0) \big( \Vert \Delta w_j\Vert_{H^{k-1}(M)} + \Vert w_j
  \Vert _{H^{1}(M)}\, \big )\\
 & \phantom{\leq} + C \epsilon_j \Vert w_j\circ \kappa_j
  \Vert_{H^{k+1}(b(r))} \\
 & \leq 2^{-j}C (1+C_0) \Vert w_j \Vert _{H^{k+1}(M)} + C \epsilon_j
  \Vert w_j\circ \kappa_j \Vert _{H^{k+1}(b(r))} \\
 & \leq C \big ( (1+C_0)^2 2^{-j} + \epsilon_j \big ) \Vert w_j\circ
  \kappa_j \Vert _{H^{k+1}(b(r))}
\end{align*}
and this is a contradiction for large $j$, since $\Vert w_j\circ
\kappa_j \Vert _{H^{k+1}(b(r))} \neq 0$. For the case when $b(r)$
has a non-empty boundary, we just include the boundary terms in
the right hand side.
\end{proof}

We finally have the following result.

\begin{theorem}\label{thm.regularity}
Let $M$ be a manifold with boundary and bounded geometry (not
necessarily with finite width). Let $k \in \NN$.  Then there exists $c
> 0$ (depending on $k$ and $(M, \partialDM)$) such that
 \begin{align*}
   \Vert u \Vert _{H^{k+1}(M)} \leq c \big(\, \Vert \Delta u\Vert
   _{H^{k-1}(M)} + \Vert u \Vert _{H^{1}(M)} + \Vert u \Vert
   _{H^{k+1/2}(\pa M)} \, \big),
  \end{align*}
for any $u \in H^{1}(M)$.
\end{theorem}

\begin{proof}
 The proof is the same as in the case of {\em compact} manifolds with
 boundary \cite{LionsMagenes1, Taylor1}, but carefully keeping track
 of the norms of the commutators. Let us review the main points,
 stressing the additional reasoning that is used in the bounded
 geometry (non-compact) case.  More to the point, we first use the
 definition of Sobolev spaces using partitions of unity
 $(\phi_\gamma)$, Proposition \ref{prop.part.unit}, to reduce our
 estimate to uniform estimates of the norms of the terms $\phi_\gamma
 \Delta u$. Since the derivatives of the partition of unity functions
 $\phi_\gamma$ are bounded, the terms $\phi_\gamma \Delta u$ can be
 uniformly estimated using estimates of $\Delta (\phi_\gamma u)$ and
 lower order norms. Lemma \ref{lemma.bd.fam2}, then gives the result.
\end{proof}

The meaning of Theorem \ref{thm.regularity} is also that if $\Delta
u \in H^{k-1}(M)$, $u \in H^{1}(M)$, and $u \in H^{k+1/2}(\pa M)$,
then, in fact, $u \in H^{k+1}(M)$.  This result shows that the domain
of $\Delta^{k}$ is contained in $H^{2k}$. 

See \cite{ADN1, Agranovich97, Browder60, Taylor1} and the references
therein for the analogous classical results for smooth, bounded
domains.  In \cite{Browder60}, more general regularity results were
obtained when $M$ is contained in a flat space. Our method of proof
is, however, different. No general results similar to the results on
the invertibility of the Laplacian (stated below) appeared in these or
other papers, though. Let us recall the following well-known
self-adjointness criteria.

\begin{remark}\label{rem.s-a}
Recall from an unnumbered corollary in \cite{SimonReed2}, Section X.1,
that every {\em closed}, symmetric (unbounded) operator $T$ that has
at least one real value in its resolvent set $\rho(T) \define \CC
\smallsetminus \sigma(T)$ is self-adjoint. Every operator $T_1$ that
has a non-empty resolvent set is closed, as one can see as a trivial
exercise. One obtains that if $T$ is symmetric and $\rho(T) \cap \RR
\neq \emptyset$, then $T$ is self-adjoint. We note that the above
mentioned corollary in \cite{SimonReed2}, Section X.1, does not state
explicitly that $T$ is required to have dense domain, but this is a
running assumption in the quoted book (see the first paragraph of
Section VIII.1 in \cite{SimonReed1}). However, it is again a trivial
exercise to show that if $T_1$ is symmetric and $\rho(T_1) \cap \RR
\neq \emptyset$, then $T_1$ is densely defined.
\end{remark}

\begin{theorem}\label{thm.Hm} 
Assume that $M$ is a manifold with bounded geometry and finite width.
Let
\begin{align*}
  \tilde{\Delta}_k-\lambda \colon H^{k+1}(M) \to H^{k-1}(M) \oplus
  H^{k+1/2}(\pa M)
\end{align*}
be given by $(\tilde{\Delta}_k - \lambda ) u \define \big( \Delta u -
\lambda u, u\vert_{\pa M} \big )$, $k \in \NN$. Then $\tilde{\Delta}_k
- \lambda$ is an isomorphism for $\Re(\lambda) < (1 + c_{M,
  \partialDM}^2)^{-1}$.  In particular, $\Delta$ is self-adjoint on
$L^2(M)$ with domain $H^2(M) \cap H^1_0(M)$.
\end{theorem}

Recall that $c_{M, \partialDM} := \infty$ if $(M, \partialDM)$ does
not satisfy the $L^2$-Poincar\'e inequality, see \eqref{eq.def.cmdm}
and Definition~\ref{def.pPoincare}.

\begin{proof}
The first part of the result follows from Theorem
\ref{thm.regularity}, the trace theorem (Theorem~\ref{thm.trace}), and
the isomorphism of Theorem \ref{thm.H1new}. To prove that $\Delta$ is
self-adjoint, we notice that, $\Delta - \lambda$ is invertible for
$\lambda\in \RR$, $\lambda<0$. Since $\Delta$ is symmetric, the result
follows from standard characterizations of self-adjoint operators, see
Remark~\ref{rem.s-a}.
\end{proof}


\subsection{Mixed boundary conditions}\label{ssec.4.3}
The case of mixed boundary conditions (thus including the case of pure
Neumann boundary conditions, i.e. $\pa M = \partialNM$) is very
similar. The main technical difference arises due to the fact that
$(H^1_D(M))^* \not\simeq H^{-1}(M)$ if $\partialNM \neq \emptyset$. To
deal with mixed boundary conditions (Dirichlet on $\partialDM$,
Neumann on $\partialNM$), one has to replace the boundary terms in $u$
with $\pa_\nu u$. Of course, the question arises how to define
$\pa_\nu u$ if $u \in H^1(M)$. This is done in a weak sense, noticing
that one has an inclusion $j_k \colon H^{k-1}(M) \oplus H^{k-1/2}(M)
\to H^1_D(M)^*$ by $(f, g)(w) = \int_M fw d\vol_g + \int_{\partialNM}
g w d\vol_{\pa g}$. If $u \in H^1_D(M)$ and $\Delta u = j_k(f, g)$,
then we say that $f$ is the ``classical'' $\Delta u$ and that
$g = \pa_\nu u$. These elementary but tedious details pertain more to
analysis than geometry, so we skip the simple proofs (the reader can
see more details in a more general setting in \cite{GN17, AGN3}),
which however follow a different line than the one indicated here. A
more substantial remark is that strongly elliptic operators with
Neumann boundary conditions still satisfy elliptic regularity, so the
proof of Lemma \ref{lemma.bd.fam2} extends to the case of Neumann
boundary conditions.

\begin{theorem}\label{thm.regularity2} Let $M$ be a manifold with boundary 
and bounded geometry (not necessarily with finite width). Let $k \in
\NN$.  Then there exists $c > 0$ (depending on $k$ and $(M,
\partialDM)$) such that
 \begin{align*}
   \Vert u \Vert _{H^{k+1}(M)} \leq c \big( \Vert \Delta u\Vert
   _{H^{k-1}(M)} + \Vert u \Vert _{H^{1}(M)} + \|u
   \|_{H^{k+1/2}(\partialDM)} + \|\pa_\nu u \|_{H^{k-1/2}(\partialNM)}
   \big),
  \end{align*}
for any $w \in H^{1}(M)$. If $\Delta u$ is in the image of $j_k$, then
$\pa_\nu u$ is defined by duality, as explained above. Otherwise, we
set $\|\pa_\nu w \|_{H^{k-1/2}(\partialNM)} = \infty$.
\end{theorem}

The following result then follows in the same way from Theorem
\ref{thm.regularity2} as Theorem \ref{thm.Hm} follows from Theorem
\ref{thm.regularity} in the previous subsection.

\begin{theorem}\label{thm.Hm2} 
Assume that $M$ is a manifold with bounded geometry and boundary, but
not necessarily of finite width. Let
\begin{align*}
  \tilde{\Delta}_k-\lambda \colon H^{k+1}(M) \to H^{k-1}(M) \oplus
  H^{k+1/2}(\partialDM) \oplus H^{k-1/2}(\partialNM)
\end{align*}
be given by $(\tilde{\Delta}_k - \lambda) u \define \big( \Delta u -
\lambda u, u\vert_{\partialDM}, \pa_\nu u\vert_{\partialNM} \big )$,
$k \in \NN$.  Then $\tilde \Delta_k - \lambda $ is an isomorphism for
$\Re(\lambda) < (1 + c_{M, \partialDM}^2)^{-1}$, where $c_{M,
  \partialDM}$ is the best constant in the Poincar\'e inequality, as
in the Introduction.  In particular, $\Delta$ is self-adjoint on
$L^2(M)$ with domain
\begin{equation*}
   \maD(\Delta ) \define \{u \in H^2(M)\vert  u = 0 \mbox{ on }
   \partialDM \mbox{ and } \pa_\nu u = 0 \mbox{ on } \partialNM\}.
\end{equation*}
\end{theorem}

Several extensions of the results of this paper are possible. Some of
them were included in the first version of the associated ArXiv
preprint, 1611.00281.v1, which had a slightly different title and will
not be published in that form. The results of that preprint 
  that were not already included in this paper, as well as others,
  will be included in two forthcoming papers \cite{GN17, AGN3}.

\def\cprime{$'$}

\end{document}